\documentclass[12pt, twoside]{article}
\usepackage{amsmath,amsthm,amssymb}
\usepackage{times}
\usepackage{enumerate}
\usepackage[colorlinks=true,citecolor=blue,linkcolor=blue]{hyperref}
\def\titlerunning#1{\gdef\titrun{#1}}
\makeatletter
\def\author#1{\gdef\autrun{\def\and{\unskip, }#1}\gdef\@author{#1}}
\def\address#1{{\def\and{\\\hspace*{18pt}}\renewcommand{\thefootnote}{}%
\footnote {#1}}%
\markboth{\autrun}{\titrun}}
\makeatother
\def\email#1{e-mail: #1}
\def\subjclass#1{{\renewcommand{\thefootnote}{}%
\footnote{\emph{Mathematics Subject Classification (2010):} #1}}}
\def\keywords#1{\par\medskip
\noindent\textbf{Keywords.} #1}

\newtheorem{theorem}{Theorem}[section]
\newtheorem{corollary}[theorem]{Corollary}
\newtheorem{lemma}[theorem]{Lemma}
\newtheorem{proposition}[theorem]{Proposition}

\theoremstyle{definition}

\newtheorem{remark}[theorem]{Remark}

\numberwithin{equation}{section}

\newcommand{\R}{\mathbb{R}}

\newcommand{\brac}[1]{\left (#1 \right )}
\newcommand{\abs}[1]{\left |#1 \right |}
\newcommand{\Ep}{\bigwedge\nolimits}

\newcommand{\lap}{\Delta }
\newcommand{\laph}{\laps{1}}
\newcommand{\aleq}{\lesssim}

\newcommand{\aeq}{\approx}
\newcommand{\Rz}{\mathcal{R}}
\newcommand{\Hz}{\mathcal{H}}
\newcommand{\laps}[1]{(-\lap) ^{\frac{#1}{2}}}
\newcommand{\lapv}{(-\lap) ^{\frac{1}{4}}}

\newcommand{\lapms}[1]{I^{#1}}

\newcommand{\dv}{\operatorname{div}}

\def\mvint_#1{\mathchoice
          {\mathop{\vrule width 6pt height 3 pt depth -2.5pt
                  \kern -8pt \intop}\nolimits_{\kern -3pt #1}}%
          {\mathop{\vrule width 5pt height 3 pt depth -2.6pt
                  \kern -6pt \intop}\nolimits_{#1}}%
          {\mathop{\vrule width 5pt height 3 pt depth -2.6pt
                  \kern -6pt \intop}\nolimits_{#1}}%
          {\mathop{\vrule width 5pt height 3 pt depth -2.6pt
                  \kern -6pt \intop}\nolimits_{#1}}}

\def\esssup{\mathop{\rm ess\,sup\,}}

\def\lip{{\rm Lip\,}}

\def\N{{\mathbb N}}

\begin{document}


\baselineskip=17pt


\titlerunning{Commutator estimates via harmonic extensions}

\title{Sharp commutator estimates\\ via harmonic extensions}

\author{Enno Lenzmann
\and 
Armin Schikorra}

\date{}

\maketitle

\address{
E. Lenzmann: University of Basel, Department of Mathematics and Computer Science, Spiegelgasse~1, CH-4051 Basel, Switzerland; \email{enno.lenzmann@unibas.ch}
\and
A. Schikorra: Mathematisches Institut, 
Abt. f\"ur Reine Mathematik, 
Albert-Ludwigs-Universit\"at Freiburg,
Eckerstra\ss{}e~1,
79104 Freiburg im Breisgau,
Germany; \email{armin.schikorra@math.uni-freiburg.de}
}



\begin{abstract}
We give an alternative proof of 
several sharp commutator estimates involving Riesz transforms, Riesz potentials, and fractional Laplacians. 
Our methods only involve harmonic extensions to the upper half-space, integration by parts, and trace space characterizations. 

The commutators we investigate are Jacobians, more generally Coifman-Rochberg-Weiss commutators, Chanillo's commutator with the Riesz potential, Coifman-McIntosh-Meyer, Kato-Ponce-Vega type commutators, the fractional Leibniz rule, and the Da~Lio-Rivi\`{e}re three-term commutators. We also give a limiting $L^1$-estimate for a double commutator of Coifman-Rochberg-Weiss-type, and several intermediate estimates. Some of the estimates obtained seem to be new or known only to some experts.

The beauty of our method is that all those commutator estimates, which are originally proven by various specific methods or by general paraproduct arguments, can be obtained purely from integration by parts and trace theorems. 
Another interesting feature is that in all these cases the ``cancellation effect'' responsible for the commutator estimate simply follows from the product rule for classical derivatives and can be traced in a precise way.

\end{abstract}
\newpage \tableofcontents
\section{Introduction: Jacobian estimates}
In this work we propose an alternative method for proving a large class of sharp and intermediate commutator estimates. The method is based on harmonic extensions to the upper half space $\R^{n+1}_+$, integration by parts, and trace space characterizations.

To illustrate the main ideas, let us first consider Jacobians $\det (\nabla u)$ for a map $u: \R^n \to \R^n$. Jacobians appear naturally in geometric analysis, they are infinitesimal deformations of the space under the map $u$, as one sees, e.g., in the change of variables formula. From the point of view of harmonic analysis, they are very special forms of commutators, as was discovered in the 1990s by Coifman, Lions, Meyer, and Semmes \cite{CLMS}. 

Fine estimates on Jacobians have proven to be crucial in particular to the theory of geometric partial differential equations. Just to name a few examples: the harmonic map equation \cite{Helein90,Riv06}, the prescribed mean curvature  equation \cite{Wente69, Heinz-1975, Bethuel-1992,Riv06}, or the conformally parametrized surface equation \cite{Mueller-Sverak-1995}. The reason is that the determinant structure acting on gradients leads to ``compensated compactness'' and ``integrability by compensation''-effects. These fine Jacobian estimates had been observed in relation with Wente's inequality \cite{Reshetnyak-1968,Wente69,BrC84,Tartar84,Mueller90}. Finally, the above mentioned seminal work of Coifman, Lions, Meyer, and Semmes \cite{CLMS} drew the connection to commutator estimates by Coifman, Rochberg, and Weiss \cite{Coifman-Rochberg-Weiss-1976} from the 1970s. In particular, the following estimate holds. 

\begin{theorem}[Jacobian estimate, \cite{CLMS,Sickel-Youssfi-1999a}]\label{th:CLMSjac}
Let $\varphi \in C_c^\infty(\R^n)$ and $u = (u^1,\ldots, u^n) \in C_c^\infty(\R^n,\R^n)$. Then
\begin{equation}\label{eq:BMOjacobian}
\int_{\R^n} \varphi\  \det(\nabla u) \aleq [\varphi]_{BMO}\ \|\nabla u\|_{L^n(\R^n)}^n.
\end{equation}
Also, the following intermediate estimate holds: let  $0 < s_i < 1$, $1 < p_i < \infty$ for $i=0,\ldots,n$ be such that \begin{equation}\label{eq:detconditions} \sum_{i=0}^n s_i = n,\quad \sum_{i=0}^{n} \frac{1}{p_i} = 1\end{equation} then
\begin{equation}\label{eq:intermediatejacobian}
 \int_{\R^n} \varphi\  \det(\nabla u) \aleq [\varphi]_{W^{s_0,p_0}}\ [u^1]_{W^{s_1,p_1}}\ \ldots\ \ [u^n]_{W^{s_n,p_n}}.
\end{equation}
\end{theorem}
For a definition of the norms we refer to Section~\ref{s:poissonchar}.
The first estimate \eqref{eq:BMOjacobian} is due to \cite{CLMS}, where it is proven that $\det(\nabla u)$ belongs to the Hardy space if $u \in W^{1,n}(\R^n)$. An alternative proof in \cite{CLMS} rewrites the Jacobian in the form of the Coifman-Rochberg-Weiss commutator \cite{Coifman-Rochberg-Weiss-1976}, see Theorem~\ref{th:CRW}. The second estimate \eqref{eq:intermediatejacobian} is due to Sickel and Youssfi \cite{Sickel-Youssfi-1999a}, who use Littlewood-Paley decomposition and paraproducts.

The main discovery of the present work is that one can obtain Theorem~\ref{th:CLMSjac} -- and several other commutator estimates (see below) -- by an integration by parts formula and trace theorems: one needs to interpret the involved functions as ${\R^n \times \{0\}}$-boundary values of harmonic functions on $\R^{n+1}_+$. We illustrate this method with a new proof first of \eqref{eq:intermediatejacobian} and then of \eqref{eq:BMOjacobian}.

\begin{proof}[Proof of intermediate estimate (\ref{eq:intermediatejacobian})]
Let $\Phi: \R^{n+1}_+ \to \R$, $U : \R^{n+1}_+ \to \R^n$ be a harmonic extension to $\R^{n+1}_+$ of $\varphi$ and $u$, respectively:
\begin{equation}\label{eq:PhiUextension}
 \begin{cases}
 \lap_{\R^{n+1}}  \Phi \equiv (\lap_x + \partial_{tt}) \Phi = 0 \quad &\mbox{in $\R^{n+1}_+$},\\
 \Phi(x,0) = \varphi(x) \quad &\mbox{in $\R^n$}.
 \end{cases}
 \quad \quad 
 \begin{cases}
 \lap_{\R^{n+1}}  U = 0 \quad &\mbox{in $\R^{n+1}_+$},\\
 U(x,0) = u(x) \quad &\mbox{in $\R^n$}.
 \end{cases}
\end{equation}
For the following to work we will choose zero-boundary data at infinity, so $\Phi$ and $U$ are given explicitly by the Poisson formula \eqref{eq:poissonkernel}, $\Phi := P_t^1 \varphi$ and $U := P_t^1 u$. 

The integration-by-parts formula\footnote{The easiest way to see this might be via Stokes' theorem for differential forms on $\R^{n+1}_+$. Observe that the boundary $\partial \R^{n+1}_+ = \R^n \times \{0\}$. Thus
\[
 \left |\int_{\R^{n+1}_+} d\Phi \wedge dU^1 \wedge dU^2 \ldots \wedge dU^n \right |=\left | \int_{\partial \R^{n}_+} \varphi\ du^1 \wedge du^2 \ldots \wedge du^n\right |.
\]
} gives us
\begin{equation}\label{eq:intbypartstrick}
 \mathcal{C} := \left |\int_{\R^n} \varphi\  \det_{n\times n}(\nabla_{\R^n} u) \right |= \left |\int_{\R^{n+1}_+} \det_{(n+1)\times (n+1)}(\nabla_{\R^{n+1}} \Phi, \nabla_{\R^{n+1}} U)\right |.
\end{equation}
This beautiful observation\footnote{Here is where the ``compensation effect'' enters: in \eqref{eq:intbypartstrick} the derivatives are uniformly distributed to all functions. Exactly this distribution of derivatives for Jacobian structures was observed by L. Tartar via the Fourier transform. He used this to give a proof for Wente's inequality \cite{Tartar84}. See also the introduction of \cite{Schikorra-SNHarmS10} where this strategy is applied to the Da~Lio-Rivi\`{e}re three-term commutator.} was used by Brezis and Nguyen \cite{Brezis-Nguyen-2011} who gave an elegant argument for estimates in terms of $[\Phi]_{\lip}$. 
However, if we aim for $W^{s,p}$-norms (and later the BMO-norm), we need to argue more carefully and distribute weights in the $(n+1)$-th direction, i.e. $t$-weights.

Namely, in view of \eqref{eq:detconditions} and H\"older's inequality,
\[
\begin{split}
 \mathcal{C}   \aleq &\brac{\int_{\R^n} \int_{0}^\infty |t^{1-\frac{1}{p_0}-s_0}\,  \nabla \Phi(x,t)|^{p_0}\ dt\ dx}^{\frac{1}{p_0}} \\
 &\cdot \brac{\int_{\R^n} \int_{0}^\infty |t^{1-\frac{1}{p_1}-s_1}\,  \nabla U^1(x,t)|^{p_1}\ dt\ dx}^{\frac{1}{p_1}}\\
 &\vdots\\
 &\cdot\brac{\int_{\R^n} \int_{0}^\infty |t^{1-\frac{1}{p_n}-s_n}\, \nabla U^n(x,t)|^{p_n}\ dt\ dx}^{\frac{1}{p_n}}. 
\end{split}
 \]
Now the trace theorems for harmonic functions, see Proposition~\ref{pr:Sobolevspaces} below, yields
\[
 \brac{\int_{\R^n} \int_{0}^\infty |t^{1-\frac{1}{p_0}-s_0} \nabla \Phi(x,t)|^{p_0}\ dt\ dx}^{\frac{1}{p_0}} \aeq [\varphi]_{W^{s_0,p_0}(\R^n)},
\]
and for $i=1,\ldots,n$,
\[
 \brac{\int_{\R^n} \int_{0}^\infty |t^{1-\frac{1}{p_i}-s_i} \nabla U^i(x,t)|^{p_i}\ dt\ dx}^{\frac{1}{p_i}} \aeq [u^i]_{W^{s_i,p_i}(\R^n)}.
\]
Thus we have shown that estimate \eqref{eq:intermediatejacobian} holds.
\end{proof}
The BMO-estimate \eqref{eq:BMOjacobian} is a little more delicate: an additional integration by parts is needed and the trace-estimates are more involved, see Proposition~\ref{pr:BMOcharacterization}.
\begin{proof}[Proof of BMO-estimate \eqref{eq:BMOjacobian}]
As in \eqref{eq:PhiUextension}, let $\Phi$ and $U$ be harmonic extensions of $\varphi$ and $u$, respectively. Again we have
\[
 \mathcal{C} := \left |\int_{\R^{n}} \det (\nabla u^1,\ldots,\nabla u^n)\, \varphi\right | = \left |\int_{\R^{n+1}_+} \det (\nabla_{\R^{n+1}} U^1,\ldots,\nabla_{\R^{n+1}} U^n, \nabla_{\R^{n+1}} \Phi) \right |.
\]
Integrating by parts in $t$-direction, we can add an additional derivative $\partial_t$, and obtain
\begin{equation}\label{eq:CLMSBMO:pt}
 \mathcal{C}=\left |\int_{\R^{n+1}_+} t\, \partial_t \det (\nabla_{\R^{n+1}} U^1,\ldots,\nabla_{\R^{n+1}} U^n, \nabla_{\R^{n+1}} \Phi) \right |.
\end{equation}
Here we used that the harmonic extensions $U$ and $\Phi$ satisfy 
\[
\lim_{t \to \infty} t\, |\nabla U(x,t)|^{n}\, |\nabla \Phi(x,t)| = \lim_{t \to 0} t\, |\nabla U(x,t)|^{n}\, |\nabla \Phi(x,t)| = 0,
\] 
see Lemma~\ref{pr:pc:maximal}. Next we claim that 
\begin{equation}\label{eq:detest}
 \mathcal{C} \aleq \int_{\R^{n+1}_+} t\ |\nabla_{\R^{n+1}} U|^{n-1} |\nabla_{\R^{n+1}} \nabla_xU|\ |\nabla_{\R^{n+1}} \Phi|.
\end{equation}
That is, we can ensure that a second derivative in $x$ hits a term different from $\Phi$ -- this can be seen as a second compensation effect. Once we have \eqref{eq:detest}, the BMO-estimate follows from trace theorems (see Proposition~\ref{pr:BMOusualestimate} for $s = 1$) and we conclude that
\[
 \mathcal{C} \aleq [\varphi]_{BMO(\R^n)}\ \|\laph u\|_{L^{n}(\R^n)}^n \aeq [\varphi]_{BMO(\R^n)}\ \|\nabla u\|_{L^{n}(\R^n)}^n.
\]
Let us prove \eqref{eq:detest}. When the derivative $\partial_t$ in \eqref{eq:CLMSBMO:pt} hits one of the $\nabla_{\R^{n+1}} U^i$, $i=1,\ldots,n$, we simply observe that the harmonicity $\partial_{tt} U = -\lap_x U$ implies
\[
 |\nabla^2_{\R^{n+1}} U |\aleq |\nabla_{\R^{n+1}} \nabla_x U |,
\]
which leads to an estimate as in \eqref{eq:detest}.

It remains to consider the case when the derivative $\partial_t$ hits $\nabla_{\R^{n+1}}\Phi$. Renaming the variables $(z_1,\ldots,z_{n+1}) = (x_1,\ldots,x_n,t)$ we have
\[
\begin{split}
 \mathcal{C}_2 :=& \left |\int_{\R^{n+1}_+} t\, \det (\nabla_{\R^{n+1}} U^1,\ldots,\nabla_{\R^{n+1}} U^n, \partial_t \nabla_{\R^{n+1}} \Phi) \right |\\
 \leq&  \sum_{i_1,\ldots,i_n = 1}^{n+1} \sum_{k=1}^{n}\left |\int_{\R^{n+1}_+} z_{n+1}\, \partial_{z_{i_1}}U^1\ldots \partial_{z_{i_n}}U^n\, \partial_{z_{n+1}} \partial_{z_k} \Phi \right |\\
 &+\sum_{i_1,\ldots,i_n = 1}^{n+1} \left |\int_{\R^{n+1}_+} z_{n+1}\, \partial_{z_{i_1}}U^1\ldots \partial_{z_{i_n}}U^n\, \partial_{z_{n+1}} \partial_{z_{n+1}} \Phi \right |\\
\end{split}
 \]
By harmonicity $\partial_{z_{n+1}z_{n+1}} \Phi = -\sum_{\ell=1}^n \partial_{z_\ell}\partial_{z_\ell} \Phi$. 
\[
\begin{split}
 \mathcal{C}_2 
 \leq&  \sum_{i_1,\ldots,i_n = 1}^{n+1} \sum_{k=1}^{n}\left |\int_{\R^{n+1}_+} z_{n+1}\, \partial_{z_{i_1}}U^1\ldots \partial_{z_{i_n}}U^n\, \partial_{z_{n+1}} \partial_{z_k} \Phi \right |\\
 &+\sum_{i_1,\ldots,i_n = 1}^{n+1} \sum_{\ell = 1}^n \left |\int_{\R^{n+1}_+} z_{n+1}\, \partial_{z_{i_1}}U^1\ldots \partial_{z_{i_n}}U^n\, \partial_{z_{\ell}} \partial_{z_{\ell}} \Phi \right |\\
\end{split}
 \]
Integrating by parts, in $z_k$ in the first term, and in $z_\ell$ in the second term, we find
\[
\begin{split}
 \mathcal{C}_2  \leq&  \sum_{i_1,\ldots,i_n = 1}^{n+1} \sum_{k=1}^{n}\left |\int_{\R^{n+1}_+} z_{n+1}\, \partial_{z_k} \brac{\partial_{z_{i_1}}U^1\ldots \partial_{z_{i_n}}U^n}\, \partial_{z_{n+1}} \Phi \right |\\
 &+\sum_{i_1,\ldots,i_n = 1}^{n+1} \sum_{\ell=1}^n \left |\int_{\R^{n+1}_+} z_{n+1}\, \partial_{z^{\ell}} \brac{\partial_{z_{i_1}}U^1\ldots \partial_{z_{i_n}}U^n}\, \partial_{z^\ell} \Phi \right |.
\end{split}
 \]
No boundary terms appear in the above integration-by-parts in $z_k$ and $z_\ell$ direction, since $k, \ell \leq n$ and the harmonic extension decays to zero sufficiently fast as $|x| \to \infty$, see Proposition~\ref{pr:pc:maximal}. We conclude that \[\mathcal{C}_2 \aleq \int_{\R^{n+1}_+} t\ |\nabla_{\R^{n+1}} U|^{n-1} |\nabla_{\R^{n+1}} \nabla_xU|\ |\nabla_{\R^{n+1}} \Phi|.\]
This establishes \eqref{eq:detest} and consequently \eqref{eq:BMOjacobian} is proven.
\end{proof}

To summarize: by an harmonic extension, integration by parts, and then trace-space characterizations we obtain the full strength of the Jacobian estimate, Theorem~\ref{th:CLMSjac}, by Coifman, Lions, Meyer, and Semmes \cite{CLMS}.

\subsection*{Outline of the paper}
In Section~\ref{s:poissonoperator} we introduce the harmonic extension via the generalized Poisson-operator.

In the remaining sections we use the ideas presented above to show several commutator estimates. Most of them have been proven before, some of the intermediate estimates seem to be new or known only to some experts. Let us remark that some of those estimates here (that we could not find in the literature) have been announced in \cite{Schikorra-epsilon}, and were proven in the arxiv-version of that paper via paraproducts.

We will treat the following:
\begin{itemize} \item Section~\ref{s:CLMSdivcurl}: The div-curl estimate by Coifman, Lions, Meyer, and Semmes. 
\item Section~\ref{s:CRW}: The Coifman-Rochberg-Weiss commutator estimate for Riesz transforms.
\item Section~\ref{s:Chanillo}: The Chanillo commutator estimate for Riesz potentials (of order $< 1$). 
\item Section~\ref{s:CMcIM}: Coifman-McIntosh-Meyer type commutator estimates.
\item Section~\ref{s:leibniz}: The fractional Leibniz rule.
\item Section~\ref{s:dalioriviere}: The Da~Lio-Rivi\`{e}re three-term commutator estimate.
\item Section~\ref{s:new}: $L^1$-estimate for a double-commutator of Coifman-Rochberg-Weiss type.
\end{itemize}
Some of the estimates seem to be new or known only to some experts; in particular we could not find in the literature the last-mentioned $L^1$-estimate of Section~\ref{s:new} and the H\"older-scale estimates of Coifman-McIntosh-Meyer type.

In the last two sections we collect the used trace-space characterization. We propose to use them as black boxes.  
In Section~\ref{s:poissonchar} we gather estimates on the Poisson operator and identification of trace spaces; in Section~\ref{s:bbestimates} we state resulting trace-inequalities. The proofs can mostly be found in the literature, in particular Stein's books. We indicate the relevant arguments in the appendix.
\subsection*{On harmonic extensions}
The technique of harmonic extensions to the upper half space and integration by parts appears in several contexts in harmonic analysis. For example, Coifman, Jones, and Semmes used a holomorphic extension \cite{Coifman-Jones-Semmes-1989} to prove $L^2$-boundedness of Cauchy integrals on Lipschitz curves. 
Chanillo provided a proof of a version of the div-curl estimate by Coifman, Lions, Meyer, and Semmes, see \cite{Chanillo-1991}, via the harmonic extension which is very similar to the above proof of \eqref{eq:BMOjacobian}. 
Also, harmonic extensions have been used for the degree-analysis of Jacobians estimates, see \cite{Bourgain-Brezis-Mironescu-2005} and the already mentioned \cite{Brezis-Nguyen-2011}; also functional calculus in the Besov spaces has been investigated \cite{Mironescu-Russ-2015}. Also commutators have been estimated via an extension argument before, see, e.g., \cite{Shen-2013} and references within. 
\subsection*{Comparison to Littlewood-Paley decomposition} There is, of course, a general technique for proving almost any commutator estimate: Littlewood-Paley decompositions and paraproducts. The advantage of the method discovered here is that the deep harmonic analysis facts are concentrated in the trace characterization results and those can be used as a black box -- see Section~\ref{s:poissonchar} and \ref{s:bbestimates}. Moreover, the \emph{cancellation effects} responsible for the commutator estimates follow from very simple product-rules and can be traced exactly. This is different from the paraproduct approach which -- while being a stronger and more general technique -- is also much more involved, seemingly messy, and less accessible to the non-expert. 

Let us also remark that the methods presented here generalize to estimates in Besov- and Triebel-Lizorkin spaces -- some trace theorems (which are estimates on Poisson-type potentials, see Section~\ref{s:poissonoperator}) actually can be seen from the identification of Besov- and Triebel spaces, see in particular the recent work by Bui and Candy \cite{BuiCandy-2015}.

\subsection*{Limits of our method} Generally, the extension method seems to be useful, if the resulting extended operator has a \emph{product rule}. As we shall see, this is the case for Riesz transforms $\Rz_i$ and $s/2$-Laplacians $\laps{s}$. For example $\lap (uv)$ is easy to compute, while $\laps{s} (uv)$ is more complicated. When the operator -- in our case the fractional Laplacian $\laps{s}$ and Riesz transforms $\Rz_i$ -- are replaced with more general operators (e.g. general multipliers, Calder\'on-Zygmund kernels), then the extension does not simplify the situation -- the extended expression may not enjoy an easily computable product rule. Then indeed the general argument of Littlewood-Paley theory and paraproducts seems more appropriate. 

Also, Fourier-transform based arguments like Littlewood-Paley theory but also our method do not seem to be well-suited to obtain pointwise commutator estimates, as e.g. the ones introduced by the second-named author in \cite{Schikorra-2011}, see also \cite{DaLio-Schikorra-palphSphere,BPSknot12, Schikorra-epsilon}.

\subsection*{Possible extensions}
This work is partially thought as an invitation to the reader: we do not expect that the above list of examples is exhaustive. It should be possible to obtain several more estimates with our method. In particular it would be interesting to see if one obtains sharp limit space estimates as in \cite{Bourgain-Li-2014,Li-2016}. 
Moreover, there are more operators that can be extended with a product rule, for example the ``Bessel'' operator $(1-\Delta)^{s/2}$, and thus for which our arguments might be useful.
We treat only differential orders $<2$, sometimes even $<1$. Since there are higher order extensions to local operators, see \cite{Yang-2013,Roncal-Stinga-2016}, it should, in principle, be also possible to obtain higher order commutator estimates.
Furthermore, it would be interesting to see how our methods performs for nonlinear commutators; for example the Calder\`on commutators, see e.g. \cite{Calderon-1965,Verdera-2001}, or nonlinear commutators related to the $p$-Laplacian, \cite{SchikorraCPDE14,Schikorra-commutators}. Also, as S. Chanillo pointed out to us after reading an earlier version of this article, it should be possible to obtain weighted norm inequalities from these methods cf. \cite{Chanillo-1991}.

\section{Harmonic extension via the Poisson operator}\label{s:poissonoperator}
For smooth, compactly supported functions $f \in C_c^\infty(\R^n,\R^m)$ the Poisson extension operator $P_t^s$ for $s > 0$ is given by
\begin{equation}\label{eq:poissonkernel}
 P_t^s f(x) := C_{n,s} \int_{\R^n} \frac{t^{s}}{\brac{|x-z|^2 + t^2}^{\frac{n+s}{2}}} f(z)\ dz.
\end{equation}
Often we will denote $F^s(x,t) := P_t^sf(x)$. For $s = 1$ the operator $P_t^1$ is the usual Poisson operator and $F^1$ is the harmonic extension of $f$ to $\R^{n+1}_+ = \R^n \times (0,\infty)$. More precisely,
\begin{equation}\label{eq:poissonpde}
 \begin{cases} 
  \lap_{\R^{n+1}} F^1(x,t) \equiv (\partial_{tt} + \lap_x) F^1(x,t) = 0 \quad &\mbox{in $\R^{n+1}_+$,}\\
  \lim\limits_{t \to 0} -\partial_t F^1(x,t) = c\, \laps{1} f(x) \quad &\mbox{on $\R^n$,}\\
  \lim\limits_{t \to 0} F^1(x,t) = f(x) \quad &\mbox{on $\R^n$},\\
  \lim\limits_{|(x,t)| \to \infty} F^1(x,t) = 0.
 \end{cases}
\end{equation}
For $s \in (0,2)$ the generalized Poisson operator $P_t^s$ the function $F(x,t) := P_t^s f(x)$ satisfies 
\begin{equation}\label{eq:spoissonpde}
 \begin{cases} 
  \dv_{\R^{n+1}}(t^{1-s} \nabla F(x,t)) = 0 \quad &\mbox{in $\R^{n+1}_+$,}\\
  \lim\limits_{t \to 0} -t^{1-s} \partial_t F(x,t) = c\, \laps{s} f(x) \quad &\mbox{on $\R^n$,}\\
  \lim\limits_{t \to 0} F(x,t) = f(x) \quad &\mbox{on $\R^n$},\\
   \lim\limits_{|(x,t)| \to \infty} F^s(x,t) = 0.
 \end{cases}
\end{equation}
The operator $P_t^s$ is sometimes called Poisson-Bessel kernel, see Marias \cite{Marias-1987}. The boundary identification $\lim\limits_{t \to 0} -t^{1-s} \partial_t F = c\, \laps{s} f$ is due to Caffarelli and Silvestre~\cite{CaffarelliSilvestre07}.
Here $\laps{s}$ denotes the fractional Laplacian on $\R^n$, defined as the operator with Fourier symbol $c\, |\xi|^s$. More precisely, denote with $\mathcal{F}$, $\mathcal{F}^{-1}$ the Fourier transforms and its inverse, respectively. Then $\laps{s}$ is defined as
\[
 \laps{s} f = \mathcal{F}^{-1}\brac{c\, |\xi|^s \mathcal{F}f}.
\]
Many function spaces involving functions $f: \R^n \to \R$ can be characterized by function spaces on the $\R^{n+1}_+$-function $F^s(x,t)$. Several of those characterizations can be found in Section~\ref{s:poissonchar}. Resulting estimates of $\R^{n+1}_+$-integrals involving Poisson extended functions can be found in Section~\ref{s:bbestimates}.

\section{The Coifman-Lions-Meyer-Semmes estimate}\label{s:CLMSdivcurl}
The estimate \eqref{eq:clmsest} below is the general div-curl estimate\footnote{In \cite{CLMS} it is shown that the div-curl-term belongs to the Hardy-space which by the Hardy-BMO-duality is equivalent to this estimate} that was proven by Coifman-Lions-Meyer-Semmes in \cite{CLMS}. Some generalizations can be found in, e.g., \cite{Strzelecki-2001,Lindberg-2016}. We could not find the intermediate estimate in the literature, although it is known to some experts. Let us also remark, that in \cite{Chanillo-1991}, Chanillo already used the harmonic extension technique to prove a version of this theorem.

\begin{theorem}[Coifman-Lions-Meyer-Semmes]\label{th:clms}
Assume that $f,\varphi \in C_c^\infty(\R^n)$, $g \in C_c^\infty(\R^n,\R^n)$ and 
\[
 \dv (g) = \sum_{i=1}^n \partial_i g_i = 0.
\]
If $1 < p_1,p_2 < \infty$, $1 \leq q_1,q_2 \leq \infty$ and $\frac{1}{p_1} + \frac{1}{p_2} = 1$, $\frac{1}{q_1} + \frac{1}{q_1} = 1$, then
\[
 \int_{\R^n} \sum_{i=1}^n\partial_i f \cdot g_i\, \varphi \aleq [\varphi]_{BMO}\, \|\nabla f\|_{L^{(p_1,q_1)}(\R^n)}\, \|g\|_{L^{(p_2,q_2)}(\R^n)}
\]
Moreover we have the following intermediate estimate. If $s_1 + s_2 + s_3 = 2$, $0 < s_1,s_2,s_3 < 1$, and $1 < p_1,p_2,p_3 < \infty$, $1 \leq q_1,q_2,q_3 \leq \infty$ such that $\frac{1}{p_1} + \frac{1}{p_2} + \frac{1}{p_3} = 1$, $\frac{1}{q_1} + \frac{1}{q_2} + \frac{1}{q_3} = 1$, then
\[
 \int_{\R^n} \sum_{i=1}^n\partial_i f \cdot g_i\ \varphi \aleq \|\laps{s_1} \varphi\|_{L^{(p_1,q_1)}}\, \|\laps{s_2} f\|_{L^{(p_2,q_2)}} \|\lapms{1-s_3} g\|_{L^{(p_3,q_3)}} 
\]
\end{theorem}

To prove this, it is convenient to use the language of differential forms. Let $f \in C_c^\infty(\R^n)$ and $g \in C_c^\infty(\R^n,\R^n)$ so that
\[
 \dv (g) = \sum_{i=1}^n \partial_i g^i = 0.
\]
If we interpret $g$ as an $(n-1)$-form, $g \in C_c^\infty(\Ep^{n-1} \R^n)$, by the Poincar\'e Lemma on differential forms, we find an $(n-2)$-form $h \in C^\infty(\Ep^{n-2}\R^n)$ so that
\[
 \sum_{i=1}^n \partial_i f\ g^i = df \wedge dh.
\]
The div-curl estimate by Coifman-Lions-Meyer-Semmes, Theorem~\ref{th:clms} is then equivalent to the following estimate. 

\begin{theorem}\label{th:clms:new}
Let $\ell \in \{0,\ldots,n-2\}$. Assume that $f \in C_c^\infty(\Ep^{\ell} \R^n)$, $h \in C_c^\infty(\Ep^{n-\ell-2} \R^n)$ and $\varphi \in C_c^\infty(\R^n)$. 
If $1 < p_1,p_2 < \infty$, $1 \leq q_1,q_2 \leq \infty$ and $\frac{1}{p_1} + \frac{1}{p_2} = 1$, $\frac{1}{q_1} + \frac{1}{q_1} = 1$, then
\begin{equation}\label{eq:clmsest}
 \int_{\R^n} df \wedge dh\, \varphi \aleq [\varphi]_{BMO}\, \|\nabla f\|_{L^{(p_1,q_1)}(\R^n)}\, \|\nabla h\|_{L^{(p_2,q_2)}(\R^n)}
\end{equation}
Moreover we have the following intermediate estimate. If $s_1 + s_2 + s_3 = 2$, $0 < s_1,s_2,s_3 < 1$, and $1 < p_1,p_2,p_3 < \infty$, $1 \leq q_1,q_2,q_3 \leq \infty$ such that $\frac{1}{p_1} + \frac{1}{p_2} + \frac{1}{p_3} = 1$, $\frac{1}{q_1} + \frac{1}{q_2} + \frac{1}{q_3} = 1$, then
\begin{equation}\label{eq:clmsestinterm}
 \int_{\R^n} df \wedge dh\, \varphi  \aleq \|\laps{s_1} \varphi\|_{L^{(p_1,q_1)}}\, \|\laps{s_2} f\|_{L^{(p_2,q_2)}} \|\laps{s_3} h\|_{L^{(p_3,q_3)}}.
\end{equation}
\end{theorem}
Let us explain the norms appearing in \eqref{eq:clmsest} and \eqref{eq:clmsestinterm}. Any $\ell$-form $f$ is of the form
\[
 f = \sum_{1 \leq i_1 < \ldots < i_\ell \leq n} f_{i_1,\ldots,i_{\ell}}\ dx^{i_1} \wedge \ldots \wedge dx^{i_\ell}.
\]
We say that $f \in C_c^\infty(\Ep^{\ell} \R^n)$, if $f_{i_1,\ldots,i_{\ell}} \in C_c^\infty(\Ep^{\ell} \R^n)$ for all $1 \leq i_1 < \ldots < i_\ell \leq n$. In this canonical way all function spaces extend to function spaces on $\ell$-forms. In particular,
\[
 \|\nabla f\|_{L^{(p,q)}(\R^n)} := \sum_{1 \leq i_1 < \ldots < i_\ell \leq n} \|
 \nabla f_{i_1,\ldots,i_{\ell}}
 \|_{L^{(p,q)}(\R^n)},
\]
and
\[
 \|\laps{s} f\|_{L^{(p,q)}(\R^n)} := \sum_{1 \leq i_1 < \ldots < i_\ell \leq n} \|
 \laps{s} f_{i_1,\ldots,i_{\ell}}
 \|_{L^{(p,q)}(\R^n)}.
\]
\begin{proof}[Proof of Theorem~\ref{th:clms:new}]
We extend $f,h,\varphi$ harmonically to $\R^{n+1}_+$, that is we let \[\Phi(x,t) := P^1_t \varphi(x),\ F(x,t) := P_t^1 f(x),\ H(x,t) := P_t^1h(x).\] 

By Stokes' theorem on differential forms,
\[
 \mathcal{C} := \left | \int_{\R^{n}} df\wedge dh\, \varphi  \right | = \left | \int_{\R^{n+1}_+} dF\wedge dH \wedge d\Phi  \right |.
 \]
The second claim, the intermediate estimate \eqref{eq:clmsestinterm}, follows from Proposition~\ref{pr:Lpqest} and
\[
 \mathcal{C} \aleq \int_{\R^{n+1}_+} |\nabla_{\R^{n+1}} F|\ |\nabla_{\R^{n+1}} H|\ |\nabla_{\R^{n+1}} \Phi|.
\]
In order to show \eqref{eq:clmsest}, we integrate-by-parts in $t$ (observe the decay as $t\to \infty$, see \eqref{eq:tinfty}),
\[
 \mathcal{C} := \left | \int_{\R^{n+1}_+} t\, \partial_t \brac{dF\wedge dH \wedge d\Phi } \right |.
\]
We claim that
\begin{equation}\label{eq:claimCLMS}
 \mathcal{C} \aleq \int_{\R^{n+1}_+} \brac{|\nabla_x \nabla_{\R^{n+1}_+} F|\, |\nabla_{\R^{n+1}_+} H| + |\nabla_{\R^{n+1}_+} F|\, |\nabla_x \nabla_{\R^{n+1}_+} H|}\ |\nabla_{\R^{n+1}} \Phi|.
\end{equation}
For this, renaming the coordinates on $\R^{n+1}_+$ from $(x_1,\ldots,x_n,t)$ to $(z_1,\ldots,z_{n+1})$
\[
 df\wedge dh \wedge d\varphi = \sum_{i,j,k=1}^{n+1} \sum_{I,J}\partial_{z_i} F_{I}\ \partial_{z_j} H_{J}\    \partial_{z_k} \Phi\ dz^i \wedge dz^I \wedge \ dz^j \wedge dz^J \wedge dz^k,
\]
where the sum is over all multiindices $I$ and $J$ which are of the form $I = (i_1, i_2, \ldots, i_\ell)$ for some $i_1 < i_2 < \ldots < i_\ell$, and $J = (j_1, j_2, \ldots, j_{n-2-\ell})$ for some $j_1 < j_2 < \ldots < j_{n-2-\ell}$.

Consequently, 
\[
\begin{split}
 \partial_{t}(dF\wedge dH \wedge d\Phi) 
 =& \sum_{i,j,k=1}^{n+1} \sum_{I,J} \partial_{z_{n+1}} \brac{\partial_{z_i} F_{I}\ \partial_{z_j} H_{J}}\    \partial_{z_k} \Phi\ dz^i \wedge dz^I \wedge \ dz^j \wedge dz^J \wedge dz^k\\
  &+ \sum_{i,j}^{n+1} \sum_{k=1}^n \sum_{I,J} \partial_{z_i} F_{I}\ \partial_{z_j} H_{J}\    \partial_{z_k} \partial_{z_{n+1}} \Phi\ dz^i \wedge dz^I \wedge \ dz^j \wedge dz^J \wedge dz^k\\
  &+ \sum_{i,j}^{n+1} \sum_{I,J} \partial_{z_i} F_{I}\ \partial_{z_j} H_{J}\    \partial_{z_{n+1}} \partial_{z_{n+1}} \Phi\ dz^i \wedge dz^I \wedge \ dz^j \wedge dz^J \wedge dz^{n+1}\\
  =:& I + II + III.
\end{split}
 \]
Observing that $\partial_{tt}F_{I} \equiv \partial_{z_{n+1}z_{n+1}}F_{I}= - \sum_{\ell = 1}^n \partial_{z_\ell}\partial_{z_\ell} F_{I}$ and likewise for $H$, for the first term $I$ we have
\[
 \left | \int_{\R^{n+1}_+} I \right |\aleq \int_{\R^{n+1}_+} \brac{|\nabla_x \nabla_{\R^{n+1}_+} F|\, |\nabla_{\R^{n+1}_+} H| + |\nabla_{\R^{n+1}_+} F|\, |\nabla_x \nabla_{\R^{n+1}_+} H|}\ |\nabla_{\R^{n+1}} \Phi|.
\]
As for the second term $II$, since $k < n+1$ the variable $z_k = x_k$, and we can integrate by parts in $z_k$,
\[
\begin{split}
 \left |\int_{\R^{n+1}_+} II \right | =& \left |\int_{\R^{n+1}_+}\sum_{i,j}^{n+1} \sum_{k=1}^n \sum_{I,J} \partial_{z_k} \brac{\partial_{z_i} F_{I}\ \partial_{z_j} H_{J}}\    \partial_{z_{n+1}} \Phi\ dz^i \wedge dz^I \wedge \ dz^j \wedge dz^J \wedge dz^k\right |\\
  \aleq &\int_{\R^{n+1}_+} \brac{|\nabla_x \nabla_{\R^{n+1}_+} F|\, |\nabla_{\R^{n+1}_+} H| + |\nabla_{\R^{n+1}_+} F|\, |\nabla_x \nabla_{\R^{n+1}_+} H|}\ |\partial_t  \Phi|.
\end{split}
 \]
Finally, for $III$, again by harmonicity of $\Phi$ we have $\partial_{z_{n+1}}\partial_{z_{n+1}}\Phi \equiv \partial_{tt}\Phi = - \sum_{\ell = 1}^n \partial_{z_\ell}\partial_{z_\ell}  \Phi$, and thus
\[
\begin{split}
 \left |\int_{\R^{n+1}_+} III \right | \leq &  \sum_{\ell = 1}^n \left |\int_{\R^{n+1}_+}\sum_{i,j}^{n+1} \sum_{k=1}^n \sum_{I,J} \partial_{z_\ell}\brac{\partial_{z_i} F_{I}\ \partial_{z_j} H_{J}}\    \partial_{z_\ell}\Phi\ dz^i \wedge dz^I \wedge \ dz^j \wedge dz^J \wedge dz^k\right |\\
  \aleq &\int_{\R^{n+1}_+} \brac{|\nabla_x \nabla_{\R^{n+1}_+} F|\, |\nabla_{\R^{n+1}_+} H| + |\nabla_{\R^{n+1}_+} F|\, |\nabla_x \nabla_{\R^{n+1}_+} H|}\ |\nabla_{x} \Phi|.
\end{split}
 \]
Consequently we have shown \eqref{eq:claimCLMS} and Theorem~\ref{th:clms:new} is proven.
\end{proof}

\section{The Coifman-Rochberg-Weiss Commutator}\label{s:CRW}
We turn to the Coifman-Rochberg-Weiss theorem \cite{Coifman-Rochberg-Weiss-1976}, more precisely the upper bound. We only prove it for Riesz transforms $(\Rz_i)_{i=1}^n$ acting on functions in $\R^n$, while the theorem is actually true for all Calder\'on-Zygmund operators. Recall that the Riesz transforms are defined as $\Rz_i = \partial_i \lapms{1}$. 

\begin{theorem}[Coifman-Rochberg-Weiss \cite{Coifman-Rochberg-Weiss-1976}]\label{th:CRW}
For any smooth and compactly supported $f,g \in C_c^\infty(\R^n)$ and any $i = 1,\ldots, n$ we define the commutator 
\[
 [\Rz_i,\varphi](g) := \Rz_i (\varphi g) - \varphi \Rz_i(g).
\]
Then, with constants depending only on $p$ and the dimension,
\begin{equation}\label{eq:CRWestimate}
 \left \|[\Rz_i,\varphi](g) \right \|_{L^p(\R^n)} \aleq [\varphi]_{BMO}\ \|g\|_{L^p(\R^n)}
\end{equation}
\end{theorem}
From the proof below, one can also obtain intermediate estimates, see Theorem~\ref{th:CMcIM}.

Actually, Coifman-Rochberg-Weiss \cite{Coifman-Rochberg-Weiss-1976} prove also the converse: if for some $p \in (1,\infty)$ \[
\left \| [\Rz_i,\varphi](g) \right \|_{L^p(\R^n)} \leq C\, \|g\|_{L^p(\R^n)}  \quad \forall i \in \{1,\ldots,n\},\ g \in L^p(\R^n),
                                                                                                         \]
then actually $\varphi \in BMO$.

Let us remark on the relation between Theorem~\ref{th:CRW} and Theorem~\ref{th:clms}. Jacobian estimates and div-curl estimates are special cases of the Coifman-Rochberg-Weiss commutator theorem. Indeed, let us illustrate this for the two-dimensional situation: take $u: \R^2 \to \R^2$ and consider the Jacobian $\det (\nabla u^1,\nabla u^2)$. The following facts are important: 
\[\Rz_2 \partial_1 f = \Rz_1 \partial_2 f\] and 
\[\int_{\R^n} g\, \Rz_i f = - \int_{\R^n} \Rz_ig\, f.\] 
Then, since $\partial_i = \Rz_i \circ \laps{1}$,
\[
\begin{split}
 &\int_{\R^2} \varphi\, \det (\nabla u^1, \nabla u^2)\\  
 =&\int_{\R^2} \varphi\, \brac{ \Rz_1 \laps{1}u^1\ \partial_2 u^2 - \Rz_2 \laps{1} u^1\ \partial_1 u^2}\\    
 =&-\int_{\R^2} \laps{1}u^1\ \brac{[\Rz_1,\varphi](\partial_2 u^2) - [\Rz_2,\varphi](\partial_1 u^2) }.
\end{split}
 \]
Similar arguments hold for determinants and div-curl products of any dimension. Thus indeed, as was discovered in \cite{CLMS}, Theorem~\ref{th:clms} is a special case of Theorem~\ref{th:CRW}.

\begin{proof}[Proof of Theorem~\ref{th:CRW}]
As before, let $\Phi$, $F$, and $G$ denote the harmonic extension to $\R^{n+1}_+$ of $\varphi$, $f$, $g$, respectively. 

With an abuse of notation we write $\tilde{\Rz}_i[F](x,t) := P^1_t \Rz_i f$. A word of warning: this object $\tilde{\Rz}_i$ acting on $\R^{n+1}$ is \emph{not} the actual Riesz transform on $\R^{n+1}$. Indeed its symbol $\sigma(\tilde{\Rz}_i)(\xi,t)$ is $\xi_i/|\xi|$ as opposed to the symbol of the $\R^{n+1}$-Riesz transform $\xi_i/(\sqrt{|\xi|^2 + t^2})$. Thus $\tilde{\Rz}_i$ is not even a H\"ormander-type multiplier operator on $\R^{n+1}$ (those multipliers are continuous away from the origin), but a rough Marcinkiewicz-multiplier (multipliers which are possibly singular at the coordinate axes).

We use the integration-by-parts formula in $t$, using the decay of the harmonic extensions from Lemma~\ref{pr:pc:maximal}. Then we have
\[
   \mathcal{C} := \left |\int_{\R^n} \varphi\, f\, \Rz_i[g] + \varphi\, \Rz_i[f]\, g \right |
   =\left |\int_{\R^{n+1}_+} t \partial_{tt} \brac{\Phi\, F\, \tilde{\Rz}_i[G] + \Phi\, \tilde{\Rz}_i[F] \, G }\right |.
\]
We claim that
\begin{equation}\label{eq:CRWwish}
 \mathcal{C} \aleq \max_{\tilde{F} \in \{F,\, \tilde{\Rz}_i F\}} \max_{\tilde{G} \in \{G,\, \tilde{\Rz}_i G\}}	\int_{\R^{n+1}_+} t |\nabla_{\R^{n+1}} \Phi|\, \brac{|\tilde{F}| |\nabla_{\R^{n+1}} \tilde{G}| + |\tilde{G}| |\nabla_{\R^{n+1}} \tilde{F}|}
\end{equation}
In words: one derivative hits $\Phi$, the other one hits $F$ or $G$. 

Once we have \eqref{eq:CRWwish}, Proposition~\ref{pr:BMOusualestimate} implies
\[
 \mathcal{C} \aleq [\varphi]_{BMO}\  \max_{\tilde{f} \in \{f,\, \Rz_i f\}} \max_{\tilde{g} \in \{g,\, \Rz_i g\}} \|\tilde{g}\|_{L^p(\R^n)}\ \|\tilde{f}\|_{L^{p'}(\R^n)} \aleq [\varphi]_{BMO}\  \|g\|_{L^p(\R^n)}\ \|f\|_{L^{p'}(\R^n)},
\]
the last inequality is the boundedness of Riesz transforms on $L^p(\R^n)$ for any $1 < p < \infty$. This estimate implies \eqref{eq:CRWestimate} by duality.

Now we establish \eqref{eq:CRWwish}. Computing the derivatives $\partial_{tt}$ we have three terms to consider: Firstly, the term
\[
\mathcal{C}_1 := \left | \int_{\R^{n+1}_+} t \brac{\partial_t\Phi\ \partial_t (F\, \tilde{\Rz}_i[G]) + \partial_t \Phi\ \partial_t (\tilde{\Rz}_i[F]\, G) } \right |
\]
can directly be estimated as in \eqref{eq:CRWwish}. Secondly, since $\partial_{tt} \Phi = -\lap_x \Phi = -\nabla_x \cdot \nabla_x \Phi$, an integration by parts in $x$-direction
\[
\mathcal{C}_2 :=  \left |\int_{\R^{n+1}_+} t \brac{\partial_{tt}\Phi\ (F\, \tilde{\Rz}_i[G] + \tilde{\Rz}_i[F]\, G) } \right |=\left |\int_{\R^{n+1}_+} t \brac{\nabla_x\Phi\ \cdot \nabla_x(F\, \tilde{\Rz}_i[G] + \tilde{\Rz}_i[F]\, G) } \right |,
\]
which again can estimated as required for \eqref{eq:CRWwish}.

Finally, it remains to find an estimate of the form \eqref{eq:CRWwish} for 
\begin{equation}\label{eq:CRWIII}
 \mathcal{C}_3 := \left |\int_{\R^{n+1}_+} t\, \Phi\ \partial_{tt} \brac{F\, \tilde{\Rz}_i[G] + \tilde{\Rz}_i[F]\, G } \right |.
\end{equation}

For this we need some rules on the interplay of Riesz transform and derivatives. Those can be computed, e.g., from the exponential representation of the Poisson potential, \[F(x,t) = \tilde{c}\, e^{-t\sqrt{-\lap}} f.\] For some constant $c \in \R$,
\begin{equation}\label{eq:ptrz1}
 \partial_t \tilde{\Rz}_i F = -c\, \partial_{x_i} F,
\end{equation}
\begin{equation}\label{eq:ptrz2}
 \lap \tilde{\Rz}_i F = c\, \partial_t \partial_{x_i} F,
\end{equation}
and
\begin{equation}\label{eq:ptrz3}
 \partial_{tt}\tilde{\Rz}_i F = -c\, \partial_{t} \partial_{x_i}  F.
\end{equation}
For sake of overview we may assume (by renormalizing $\tilde{\Rz}_i$) that $c = 1$.

One cancellation effect for the estimate \eqref{eq:CRWIII} appears here:
\[
 \partial_{t} \brac{F\partial_{t} \tilde{\Rz}_i[G] + \partial_{t} \tilde{\Rz}_i[F]\, G} = -\partial_{x_i} \partial_t (FG).
\]
Moreover, using harmonicity, $\partial_{tt} F = -\lap_x F$, (observe that everything commutes with $\tilde{\Rz}_{i}$)
\[
\begin{split}
 \partial_{t} \brac{\partial_{t}F\, \tilde{\Rz}_i[G] +  \tilde{\Rz}_i[F]\, \partial_{t}G}
 =& -\lap_x F\, \tilde{\Rz}_i[G] -  \tilde{\Rz}_i[F]\, \lap_x G\\
 &+\lap_x \tilde{\Rz}_i[F]\, G + F\, \lap_x \tilde{\Rz}_i[G]\\
 &-\partial_{x_i}\brac{\partial_{t}F\, G + F\, \partial_{t}G},
\end{split}
\]
and with a second cancellation effect
\[
\begin{split}
 =& -\nabla_x \cdot \brac{\nabla_x F\, \tilde{\Rz}_i[G] -  \tilde{\Rz}_i[F]\, \nabla_x G}\\
 &+\nabla_x\cdot \brac{\nabla_x \tilde{\Rz}_i[F]\, G + F\, \nabla_x \tilde{\Rz}_i[G]}\\
 &-\partial_{x_i}\brac{\partial_{t}F\, G + F\, \partial_{t}G}.
\end{split}
 \]
Thus, we have shown that
\[
\begin{split}
\partial_{tt} \brac{F\, \tilde{\Rz}_i[G] + \tilde{\Rz}_i[F]\, G }
=&-\partial_{x_i} \partial_t (FG)\\
 & -\nabla_x \cdot \brac{\nabla_x F\, \tilde{\Rz}_i[G] -  \tilde{\Rz}_i[F]\, \nabla_x G}\\
 &+\nabla_x\cdot \brac{\nabla_x \tilde{\Rz}_i[F]\, G + F\, \nabla_x \tilde{\Rz}_i[G]}\\
 &-\partial_{x_i}\brac{\partial_{t}F\, G + F\, \partial_{t}G}.
\end{split}
\]
Plugging this into \eqref{eq:CRWIII} and performing an integration by parts in $x$-direction (no boundary terms appear in $x$-direction), we see the estimate of the form \eqref{eq:CRWwish}.
\end{proof}

\section{Chanillo-type commutator of Riesz Potentials}\label{s:Chanillo}
For $s \in (0,1)$ we also obtain an extension of the results of Coifman-Rochberg-Weiss to Riesz potentials. In \cite{Chanillo-1982}, Chanillo showed the following theorem on commutators of Riesz potential and pointwise multiplication,
\[
 [\lapms{s},\varphi](f) := \lapms{s} (\varphi f) - \varphi \lapms{s} f.
\]

\begin{theorem}[Chanillo]\label{th:chanillo}
Let $s \in (0,n)$ then for any $f$, $\varphi \in C_c^\infty(\R^n)$,
\[
 \|[\lapms{s},\varphi](f)\|_{L^q(\R^n)} \aleq [\varphi]_{BMO}\ \|f\|_{L^p(\R^n)},
\]
where $1 < p < \frac{n}{s}$ and
\begin{equation}\label{eq:chanillopq}
 \frac{1}{q} = \frac{1}{p}-\frac{s}{n}.
\end{equation}
\end{theorem}
By duality, setting $u := \lapms{s} f$, Theorem~\ref{th:chanillo} is a consequence of the following proposition. It is stated for $s \in [0,1)$. With little extra work (iterating the integration-by-parts procedure) one can extend this to $s \in [0,2)$. For higher order $s$ one first needs a suitable higher-order extension replacing the one by Caffarelli and Silvestre \cite{CaffarelliSilvestre07}. This is done in \cite{Yang-2013,Roncal-Stinga-2016}. So we think it is likely to obtain the full Theorem~\ref{th:chanillo} with this method, but we will make no attempt to prove this here.
\begin{proposition}
Let $s \in [0,1)$, and $p$, $q$ as in Theorem~\ref{th:chanillo}, $q' = \frac{q}{q-1}$ then
\[
\int_{\R^n} \brac{\laps{s} u\, v - u \laps{s} v} \varphi \aleq [\varphi]_{BMO}\ \|\laps{s} u\|_{L^p}\ \|\laps{s} v\|_{L^{q'}} .
\]
\end{proposition}
\begin{proof}
Let $U(x,t) := P_t^{s} u(x)$, $V(x,t) := P_t^{s} v(x)$, $\Phi(x,t) := P_t^s \varphi(x)$ the Caffarelli-Silvestre extension, as in \eqref{eq:spoissonpde}.
Then with the integration-by-parts formula in $t$,
\[
\mathcal{C} := \left |\int_{\R^n} \brac{\laps{s} u\, v - u \laps{s} v} \varphi \right |= c\left |\int_{\R^{n+1}_+} \partial_t \brac{t^{1-s}\brac{\partial_t U\, V - U\, \partial_t V} \Phi}  \right |.
\]
By two cancellation effects and since $\partial_{t} (t^{1-s} \partial_t U) = -t^{1-s} \lap_x U$,
\[
 \partial_t \brac{t^{1-s}\brac{\partial_t U\, V - U\, \partial_t V}} = -t^{1-s} \brac{\lap_x U\, V - U\, \lap_x V} = -t^{1-s} \nabla_x \cdot \brac{\nabla_x U\, V - U\, \nabla_x V}.
\]
Thus,
\[
 \mathcal{C} \aleq \left |\int_{\R^{n+1}_+} t^{1-s}\brac{\nabla_x U\, V - U\, \nabla_x V} \nabla_x \Phi \right | + \left |\int_{\R^{n+1}_+} t^{1-s}\brac{\partial_t U\, V - U\, \partial_t V} \Phi_t \right |
\]
By our assumption $s < 1$, and as we shall see below the first term is already in a good shape and can be estimated by Proposition~\ref{pr:BMOusualestimate}.
The second term needs one more step, because with $\partial_t U$ and Proposition~\ref{pr:BMOusualestimate} we only get an estimate in terms of $\laps{\nu} u$ for $\nu < s$. 
So we use again the integration-by-parts in $t$,
\[
  \int_{\R^{n+1}_+} t^{1-s}\brac{\partial_t U\, V - U\, \partial_t V} \Phi_t  = \frac{1}{s} \int_{\R^{n+1}_+} t^s\, \partial_{t} \brac{ t^{1-s}\brac{\partial_t U\, V - U\, \partial_t V} t^{1-s}\Phi_t} .
\]
Again we observe a cancellation, 
\[
  t^s\partial_{t} \brac{ t^{1-s}\brac{\partial_t U\, V - U\, \partial_t V} t^{1-s}\Phi_t} = -t^{2-s} \brac{\lap_x U\, V - U\, \lap_x V} \Phi_t - t^{2-s}\brac{\partial_t U\, V - U\, \partial_t V} \lap_x \Phi.
\]
With yet another integration by parts in $x$-direction, since  $\lap_x = \nabla_x \cdot \nabla_x$ we arrive at
\[
\begin{split}
 \mathcal{C} \aleq &\int_{\R^{n+1}_+} t^{1-s}\brac{|\nabla_x U|\ |V| + |U|\, |\nabla_x V|} |\nabla_{\R^{n+1}}\Phi| \\
 &+ \int_{\R^{n+1}_+} t^{2-s} \brac{|\nabla_{\R^{n+1}} \nabla_{x}U|\, | V| + |U|\, |\nabla_{\R^{n+1}} \nabla_{x} V|}\ |\nabla_{\R^{n+1}} \Phi|\\
 &+ \int_{\R^{n+1}_+} t^{2-s} |\nabla_{\R^{n+1}} U|\, | \nabla_{\R^{n+1}} V|\ |\nabla_{\R^{n+1}} \Phi|.
\end{split}
 \]
By Proposition~\ref{pr:BMOusualestimate}, 
\[
\mathcal{C} \aleq [\varphi]_{BMO}\ \brac{\|\laps{s} u\|_{L^p}\ \|v\|_{L^{p'}} + \|u\|_{L^q}\ \|\laps{s} v\|_{L^{q'}}}.
\]
We conclude by the relation between $p$ and $q$ \eqref{eq:chanillopq} and Sobolev-inequality:
\[
 \|v\|_{L^{p'}} \aleq \|\laps{s} v\|_{L^{q'}}, \quad \|u\|_{L^{q}} \aleq \|\laps{s} u\|_{L^{p}}.
\]
\end{proof}

\section{Coifman-McIntosh-Meyer type commutator estimate}\label{s:CMcIM}
Now we treat commutators in terms of H\"older norms, namely we consider
\[
 [\laps{s}, g](f) = \laps{s}(gf) - g \laps{s}f,
\]
and its (nontrivial) zero-order version
\[
 [\Rz_i, g](f) = \Rz_i(gf) - g \Rz_if.
\]
The estimate \eqref{eq:lapsgfglapsflip} is probably most close to the Coifman-Meyer commutator estimates and Kato-Ponce type estimates, see \cite{Coifman-Meyer-1986,Coifman-McIntosh-Meyer-1982,Kato-Ponce-1988}. The estimates \eqref{eq:lapsgfglapsfhoel} and \eqref{eq:rieszgf} seem to be new.
The limit case $\sigma = 0$ for \eqref{eq:rieszgf} is the Coifman-Rochberg-Weiss theorem, Theorem~\ref{th:CRW}.

\begin{theorem}\label{th:CMcIM}
Let $s \in (0,1]$ and $p \in (1,\infty)$. Then,
\begin{equation}\label{eq:lapsgfglapsflip}
\| [\laps{s}, g](f)\|_{L^p(\R^n)} \aleq [g]_{\lip} \|\lapms{1-s} f \|_{L^p(\R^n)}.
\end{equation}
More generally, for $\sigma \in [s,1]$, $f,g \in C_c^\infty(\R^n)$,
\begin{equation}\label{eq:lapsgfglapsfhoel}
\| [\laps{s}, g](f)\|_{L^p(\R^n)} \aleq \brac{[\laps{\sigma} g]_{BMO} + [g]_{C^\sigma}} \|\lapms{\sigma-s} f \|_{L^p(\R^n)}.
\end{equation}
Also, for $q_1,q_2,p \in (1,\infty)$, $\frac{1}{q_1} + \frac{1}{q_2} = \frac{1}{p}$, $\sigma \in [s,1)$,
\begin{equation}\label{eq:lapsgfglapsfintermediate}
\| [\laps{s}, g](f)\|_{L^p(\R^n)} \aleq \|\laps{\sigma} g\|_{L^{q_1}(\R^n)}\ \|\lapms{\sigma-s} f \|_{L^{q_2}(\R^n)}.
\end{equation}
For $\sigma < 1$, any $i =1,\ldots,n$,
\begin{equation}\label{eq:rieszgf}
 \| [\Rz_i, g] f \|_{L^p(\R^n)} \aleq [\laps{\sigma} g]_{BMO}\,  \|\lapms{\sigma} f\|_{L^p(\R^n)}.
\end{equation}
Also, for $q_1,q_2,p \in (1,\infty)$, $\frac{1}{q_1} + \frac{1}{q_2} = \frac{1}{p}$, $\sigma \in [0,1)$,
\begin{equation}\label{eq:rieszgfintermediate}
 \| [\Rz_i, g] f \|_{L^p(\R^n)} \aleq \|\laps{\sigma} g\|_{L^{q_1}(\R^n)}\ \|\lapms{\sigma} f\|_{L^{q_2}(\R^n)}.
\end{equation}
\end{theorem}
For $n=1$ and $s = 1$, the commutator in \eqref{eq:lapsgfglapsflip} is also called the first Calder\'{o}n commutator \cite{Calderon-1965}.
\begin{proof}[Proof of \eqref{eq:lapsgfglapsflip}, \eqref{eq:lapsgfglapsfhoel}, and \eqref{eq:lapsgfglapsfintermediate}]
Let again $F(x,t) := P^s_t f(x)$, and $G,H$ likewise be the $P_t^s$-extension of $g$, $h$. Integration by parts gives
\[
\begin{split}
 \mathcal{C} := &\int_{\R^n} g\, f\, \laps{s} h - g\, \laps{s} f\, h\\
=& \int_{\R^{n+1}_+}\partial_t \brac{G \brac{ F\, t^{1-s}\partial_t H - t^{1-s}\partial_t F\, H}}.
\end{split}
\]
We claim that
\begin{equation}\label{eq:CMgoal}
\begin{split}
 \mathcal{C}  \aleq& \int_{\R^{n+1}_+}t^{2-s}\, |\nabla_{\R^{n+1}} \nabla_x G| \brac{|\nabla_{\R^{n+1}} F|\, |H| + |\nabla_{\R^{n+1}} H|\, |F|}\\
& +\int_{\R^{n+1}_+}t^{2-s}\, |\nabla_x G|\, |\nabla_{\R^{n+1}} F|\, |\nabla_{\R^{n+1}} H|.
\end{split}
 \end{equation}
Once we confirm this, we argue with Proposition~\ref{pr:Hoelderusualestimate} for \eqref{eq:lapsgfglapsflip}, \eqref{eq:lapsgfglapsfhoel} and with Proposition~\ref{pr:Lpqest} for \eqref{eq:lapsgfglapsfintermediate}. Taking in the resulting estimate the supremum over all $h$ with $\|h\|_{L^{p'}} \leq 1$ we obtain \eqref{eq:lapsgfglapsflip}, \eqref{eq:lapsgfglapsfhoel}  and \eqref{eq:lapsgfglapsfintermediate}, respectively. 

It remains to show \eqref{eq:CMgoal}. Observe a first cancellation
\[
\begin{split}
 &\partial_t \brac{G \brac{ F\, t^{1-s}\partial_t H - t^{1-s}\partial_t F\, H}}\\
= &{t^{1-s} G_t \brac{ F\, H_t - F_t\, H}} + G \brac{ F\, \partial_t (t^{1-s}\partial_t H) - \partial_t(t^{1-s}\partial_t F)\, H},
\end{split}
\]
and since $\partial_t (t^{1-s} F_t) = -t^{1-s} \lap_x F$,
\[
= {t^{1-s} G_t \brac{ F\, H_t - F_t\, H}} - G t^{1-s}\brac{ F\, \lap_x H - \lap_x F\, H}.
\]
With another cancellation in the second term,
\[
= {t^{1-s} G_t \brac{ F\, H_t - F_t\, H}} - G t^{1-s} \nabla_x \cdot \brac{ F\, \nabla_x H - \nabla_x F\, H}\\
\]
Using integration-by-parts in $x$ we decompose $\mathcal{C} = \mathcal{C}_1 + \mathcal{C}_2$ with
\[
\mathcal{C}_1  := \int_{\R^{n+1}_+}t^{1-s} G_t \brac{ F\, H_t - F_t\, H }, \quad \mathcal{C}_2 := \int_{\R^{n+1}_+}t^{1-s} \nabla_{x}G \cdot \brac{ F\, \nabla_{x} H - \nabla_{x} F\, H }.
\]
As for the second term, integration by parts in $t$-direction, gives 
\[
 \mathcal{C}_2 = -\frac{1}{2-s} \int_{\R^{n+1}_+}t^{2-s} \partial_t \brac{\nabla_{x}G \cdot \brac{ F\, \nabla_{x} H - \nabla_{x} F\, H }}.
\]
Now the only term that is not already of a form needed for \eqref{eq:CMgoal} is the case where the $\partial_t$ hits $\nabla_x H$ or $\nabla_x F$. But then we perform another integration-by-parts in $x$-direction,
\[
\begin{split}
 &\int_{\R^{n+1}_+}t^{2-s}\ \nabla_{x}G \cdot \brac{ F\, \nabla_{x} H_t - \nabla_{x} F_t\, H }\\
 =&-\int_{\R^{n+1}_+}t^{2-s}\ \lap_{x}G \brac{ F\, H_t - F_t\, H } -\int_{\R^{n+1}_+}t^{2-s}\ \nabla_{x}G \cdot \brac{ \nabla_{x}F\,  H_t - F_t\, \nabla_{x} H }.
\end{split}
 \]
This is clearly of the form needed for \eqref{eq:CMgoal}. Thus $\mathcal{C}_2$ is estimated.

As for $\mathcal{C}_1$, integration by parts tells us
\[
\mathcal{C}_1  = \int_{\R^{n+1}_+}t^{1-s} G_t \brac{ F\, H_t - F_t\, H }, = -\frac{1}{s}\int_{\R^{n+1}_+}t^{s} \partial_t\brac{t^{1-s} G_t \brac{ F\, t^{1-s} H_t - t^{1-s} F_t\, H }}.
\]
Now, a second cancellation happens, since $F_t\, t^{1-s} H_t - t^{1-s} F_t\, H_t = 0$,
\[
 = \frac{1}{s}\int_{\R^{n+1}_+}t^{2-s}  \lap_x G \brac{ F\, H_t - F_t\, H } +\frac{1}{s}\int_{\R^{n+1}_+}t^{2-s}  G_t \brac{ F\, \lap_x H - \lap_x F\, H }.	
\]
and for the second term a further cancellation $0 = \nabla F\cdot \nabla H - \nabla F\cdot \nabla H$,
\[
 = \frac{1}{s}\int_{\R^{n+1}_+}t^{2-s}  \lap_x G \brac{ F\, H_t - F_t\, H } +\frac{1}{s}\int_{\R^{n+1}_+}t^{2-s}  G_t \nabla \cdot \brac{ F\, \nabla H - \nabla F\, H }.	
\]
Integrating by parts in $x$ we obtain an estimate of the form \eqref{eq:CMgoal}.
\end{proof}

\begin{proof}[Proof of \eqref{eq:rieszgf} and \eqref{eq:rieszgfintermediate}]
Let $F,G, \Phi$ be the harmonic extension of $f,g, \varphi$. By duality it suffices to show
\[
\mathcal{C} := \int_{\R^n}  g\ f \Rz_i[\varphi] + g \Rz_i [f] \varphi  \aleq \begin{cases}
                                                                              [g]_{C^\sigma}\ \brac{[\laps{\sigma} g]_{BMO} + [D^{\sigma} g]_{BMO}}  \|\lapms{\sigma} f\|_{L^p(\R^n)}\ \|\varphi\|_{L^{p'}(\R^n)}\\
                                                                              \|\laps{\sigma} g\|_{L^{q_1}(\R^n)}\ \|\lapms{\sigma} f\|_{L^{q_2}(\R^n)}\ \|\varphi\|_{L^{p'}(\R^n)}.
                                                                             \end{cases}
\]
We estimated $\mathcal{C}$ in the proof of Theorem~\ref{th:CRW} (note that the role of $\Phi$ and $G$ are exchanged there). Setting
\[
 \mathcal{C}  := \int_{\R^{n+1}_+} t\, \partial_{tt}  \brac{G\, F\, \tilde{\Rz}_i\Phi + G\, \tilde{\Rz}_iF\, \Phi },
\]
and we have
\begin{equation}\label{eq:reiszgf:goal}
 \mathcal{C} \aleq \max_{\tilde{F} \in \{F,\, \tilde{\Rz}_i F\}} \max_{\tilde{\Phi} \in \{\Phi,\, \tilde{\Rz}_i \Phi\}}	\int_{\R^{n+1}_+} t\, |\nabla_{\R^{n+1}} G| \brac{ |\nabla_{\R^{n+1}} F| |\Phi| + |F| |\nabla_{\R^{n+1}} \Phi|}
\end{equation}
The claim follows now from Proposition~\ref{pr:Hoelderusualestimate} and Proposition~\ref{pr:Lpqest}.
\end{proof}

\section{Fractional Leibniz rule}\label{s:leibniz}
The Leibniz rule implies 
\[
                                                H_{\nabla}(f,g) := \nabla (fg) - \nabla f\, g - f\, \nabla g \equiv 0.
\]
If one replaces $\nabla$ with $\laps{s}$ and defines for $s > 0$,
\[
H_s(f,g) := \laps{s} (fg) - \laps{s}f\, g - f\, \laps{s} g,
\]
it may be that $H_s(f,g) \neq 0$. For example,
\[
 H_2(f,g) = 2\, \nabla f\cdot \nabla g.
\]
However, there are the so-called \emph{fractional} Leibniz-rules, such as \eqref{eq:interleibniz} below. Originally, they are due to Kenig-Ponce-Vega \cite{Kenig-Ponce-Vega-1993}, see also \cite{Grafakos-Maldonado-Naibo-2014,Bourgain-Li-2014}. Our extension method also shows a limit estimate, \eqref{eq:limitleibniz}, which was announced in \cite[(5.29)]{Schikorra-epsilon} and proven (in the arxiv-version of that paper) with para-product arguments.

\begin{theorem}\label{th:HsBMOest}
For any $s \in (0,1]$, $p \in (1,\infty)$, $f$, $\varphi \in C_c^\infty(\R^n)$,
\begin{equation}\label{eq:limitleibniz}
 \|H_{s} (f, \varphi)\|_{L^p(\R^n)} \aleq \|\laps{s} f\|_{L^p(\R^n)}\ [\varphi]_{BMO}
\end{equation}
Also we have an intermediate estimate: for any $t \in (0,s)$, $p,p_1,p_2 \in (1,\infty)$, $q,q_1,q_2 \in [1,\infty]$ such that
\[
 \frac{1}{p} = \frac{1}{p_1} + \frac{1}{p_2}, \quad \frac{1}{q} = \frac{1}{q_1} + \frac{1}{q_2},
\]
it holds that
\begin{equation}\label{eq:interleibniz}
\|H_{s} (f, \varphi)\|_{L^{(p,q)}(\R^n)} \aleq \|\laps{s-t} f\|_{L^{(p_1,q_1)}(\R^n)}\ \|\laps{t} \varphi\|_{L^{(p_2,q_2)}(\R^n)},
\end{equation}
\end{theorem}
\begin{proof}
We only show the BMO-estimate, the intermediate estimate follows with the same argument using Proposition~\ref{pr:Lpqest} instead of Proposition~\ref{pr:BMOusualestimate}.

By duality we need to show
\[
 \mathcal{C} := \left |\int_{\R^n} f\ \varphi\ \laps{s} g - \laps{s} f\ \varphi\ g- f\ \laps{s} \varphi\ g\right | \aleq [\varphi]_{BMO}\ \|\laps{s} f\|_{L^p(\R^n)}\ \|g\|_{L^{p'}(\R^n)}.
\]
Letting $F(x,t) := P_t^s f(x)$, $G(x,t) := P_t^s g(x)$ and $\Phi(x,t) := P_t^s \varphi(x)$, an integration by parts in $t$ gives
\[
 \mathcal{C} \aleq \left |\int_{\R^{n+1}_+} \partial_t \brac{t^{1-s}F\ \Phi\ \partial_t G - t^{1-s}\partial_t F\ \Phi\ G- t^{1-s} F\ \partial_t \Phi\ G}\right |.
\]
We compute,
\[
\begin{split}
 &\partial_t \brac{t^{1-s}F\ \Phi\ \partial_t G - t^{1-s}\partial_t F\ \Phi\ G- t^{1-s} F\ \partial_t \Phi\ G}\\
  =&t^{1-s}\, \brac{\lap_x (F \Phi)\ G-F\ \Phi\ \lap_x G }\\
  &- 2 \nabla_x F\cdot \nabla_x \Phi\ G-2t^{1-s}\, \partial_t F\ \partial_t \Phi\ G.
\end{split}
 \]
The first term integrates to zero when integrating in $x$,
\[
 \int_{\R^{n+1}_+} t^{1-s}\, \brac{\lap_x (F \Phi)\ G-F\ \Phi\ \lap_x G } = 0.
\]
So we have 
\[
 \mathcal{C} \aleq \left |\int_{\R^{n+1}_+}  t^{1-s}\nabla_x F\cdot \nabla_x \Phi\ G\right |+\left |\int_{\R^{n+1}_+} t^{1-s}\, \partial_t F\ \partial_t \Phi\ G \right |.
\]
For $s < 1$, the first term already can be estimated by Proposition~\ref{pr:BMOcharacterization}
\[
\mathcal{C}_1 := \left |\int_{\R^{n+1}_+}  t^{1-s}\nabla_x F\cdot \nabla_x \Phi\ G\right | \aleq [\varphi]_{BMO}\ \|\laps{s} f\|_{L^{p}}\ \| g\|_{L^{p'}}.
\]
For $s = 1$, by another integration-by-parts in $t$-direction,
\[
 \mathcal{C}_1 = \left |\int_{\R^{n+1}_+}  t\, \partial_t (\nabla_x F\cdot \nabla_x \Phi\ G)\right | \aleq \int_{\R^{n+1}_+}  t\, |\nabla_{\R^{n+1}} \Phi|\, \brac{|\nabla_{\R^{n+1}}\nabla_x F| \ |G| + |\nabla_{\R^{n+1}} F| \ |\nabla_{\R^{n+1}} G| }.
\]
Indeed, the only term not immediately in this constellation can be transformed into the right form by an integration-by-parts in $x$-direction 
\[
 \int_{\R^{n+1}_+}  t\, \nabla_x F\cdot \nabla_x \partial_t \Phi\ G =-\int_{\R^{n+1}_+}  t\, \nabla_x \cdot (\nabla_x F\  G)\ \partial_t \Phi.
\]
Thus also for $s=1$, again with the help of Proposition~\ref{pr:BMOcharacterization}, 
\[
 \mathcal{C}_1 \aleq [\varphi]_{BMO}\ \|\laps{s} f\|_{L^{p}}\ \| g\|_{L^{p'}}.
\]
For the remaining term
\[
 \mathcal{C}_2 := \left |\int_{\R^{n+1}_+} t^{1-s}\, \partial_t F\ \partial_t \Phi\ G \right |
\]
with an integration-by-parts in $t$-direction,
\[
\mathcal{C}_2= \frac{1}{s}\left |\int_{\R^{n+1}_+} t^s \partial_t \brac{t^{1-s}\, \partial_t F\ t^{1-s}\partial_t \Phi\ G} \right |.
\]
In view of $\partial_t (t^{1-s} \partial_t F) = -ct^{1-s} \lap_x F$, this can be estimated by
\[
\mathcal{C}_2 \aleq \int_{\R^{n+1}_+}  t^{2-s}\, |\nabla_{\R^{n+1}} \Phi|\, \brac{|\nabla_{\R^{n+1}}\nabla_x F| \ |G| + |\nabla_{\R^{n+1}} F| \ |\nabla_{\R^{n+1}} G| }.
\]
Indeed, the only term not in this form can be treated as above,
\[
 \int_{\R^{n+1}_+} t^{2-s}\, \partial_t F\ t^{1-s}\lap_x \Phi\ G = -\int_{\R^{n+1}_+} t^{2-s}\, \nabla_x \Phi \cdot \nabla_x \brac{\partial_t F\  G}.
\]
We conclude with Proposition~\ref{pr:BMOcharacterization}.
\end{proof}

\section{Da~Lio-Rivi\`{e}re three-term commutator}\label{s:dalioriviere}
Another limit-space estimate of the three-term commutator $H_s(f,g)$ from Section~\ref{s:leibniz}, 
\[
H_s(f,g) := \laps{s} (fg) - \laps{s}f\, g - f\, \laps{s} g,
\]
is due to Da~Lio and Rivi\`{e}re, \cite{DR1dSphere,DR1dMan}. 
They showed that $\lapv H_{\frac{1}{2}}(f,g)$ appears as a natural replacement for the Jacobian structure for 1/2-harmonic maps. See also \cite{Schikorra-SNHarmS10,DndMan,Schikorra-epsilon} for higher order analogues and extensions. 

In \cite[Theorem 1.2]{DR1dSphere} the following three-term commutator estimate is proven.
\begin{theorem}[Da~Lio-Rivi\`{e}re \cite{DR1dSphere}]\label{th:DR}
For $a,b \in C_c^\infty(\R^n)$,
\[
\|(-\lap)^{\frac{1}{4}} H_{\frac{1}{2}} (a,b) \|_{\mathcal{H}^1(\R^n)} \aleq \|\lapv a\|_{L^2(\R^n)}\ \|\lapv b\|_{L^2(\R^n)}.
\]
Here, $\mathcal{H}^1(\R^n)$ denotes the Hardy-space.
\end{theorem}
For the proof, Da~Lio and Rivi\`ere used the theory of Triebel-Lizorkin spaces and paraproducts. Extending their techniques, the following was shown in \cite{Schikorra-epsilon} (for a proof see the arxiv-version). Again the original proof requires a lengthy computation with Triebel spaces and paraproducts. In particular the $s=1$-case was somewhat unexpected and required special care. Now it just follows from integration by parts.
\begin{theorem}
Let $s \in (0,1]$, $p \in (1,\infty)$, $p' = \frac{p}{p-1}$, $q \in [1,\infty]$, $q' = \frac{q}{q-1} \in [1,\infty]$. Then for any $a,b \in C_c^\infty(\R^n)$,
\begin{equation}\label{eq:Hsbmoest}
\int_{\R^n} H_s(a,b)\, \laps{s} \varphi \aleq [\varphi]_{BMO}\, \|\laps{s} a\|_{L^{(p,q)}(\R^n)}\|\laps{s} b\|_{L^{(p',q')}(\R^n)}.
\end{equation}
In particular, by the duality of Hardy-space $\mathcal{H}^1$ and BMO,
\[
\|\laps{s} \brac{ H_s(a,b)}\|_{\mathcal{H}^1} \aleq \|\laps{s} a\|_{L^{(p,q)}}\|\laps{s} b\|_{L^{(p',q')}}.
\]
\end{theorem}
\begin{proof}[Proof for $s=1$]
Assume that $s = 1$, and let $A$, $B$, $\Phi$ be the harmonic extensions as in \eqref{eq:poissonpde} of $a$, $b$, $\varphi$, respectively.

We set
\[
 \mathcal{C} := \left |\int_{\R^n} H_1(a,b) \laps{1} \varphi \right |= \left |\int_{\R^n}  a\ b\ (-\lap) \varphi - \laps{1} a\ b\ \laps{1}\varphi - a\ \laps{1} b\ \laps{1}\varphi \right |.
\]

We show the following estimate from which the claim follows via Proposition~\ref{pr:BMOusualestimate}.
\begin{equation}\label{eq:hopeful}
 |\mathcal{C}| \aleq \int_{\R^{n+1}_+} t |\nabla_{\R^{n+1}} \Phi|\, \brac{|\nabla_{\R^{n+1}} A|\, |\nabla_{\R^{n+1}}^2 B| + |\nabla_{\R^{n+1}}^2 A|\, |\nabla_{\R^{n+1}} B| }. 
\end{equation}

To obtain \eqref{eq:hopeful} we use the integration-by-parts in $t$,
\[
\begin{split}
\mathcal{C}=&\left |\int_{\R^{n+1}_+} t\, \partial_{tt} \brac{A\, B\, \partial_{tt} \Phi - A\, \partial_t B\, \partial_t \Phi- \partial_t A\, B\, \partial_t \Phi}\right |\\
=&\left |\int_{\R^{n+1}_+} t\, \partial_{tt} \brac{(A\, B)\, \partial_{tt} \Phi - \partial_t (A\, B)\, \partial_t \Phi} \right |.
\end{split}
\]
In the next step a cancellation occurs. By the product-rule for $\partial_t$,
\[
=\left |\int_{\R^{n+1}_+} t\, \partial_{t} \brac{(A\, B)\, \partial_{ttt} \Phi - \partial_{tt} (A\, B)\, \partial_t \Phi} \right |.
\]
Due to the harmonicity of the extensions \eqref{eq:poissonpde} we may replace $\partial_{tt} \Phi$ by $-\lap_x \Phi$, and then use integration by parts on $\lap_x$ which does not give boundary values since it is in tangential direction,
\[
\begin{split}
=&\left |\int_{\R^{n+1}_+} t\, \partial_{t} \brac{(A\, B)\, (-\lap_x) \partial_{t} \Phi - \partial_{tt} (A\, B)\, \partial_t \Phi} \right |\\
=&\left |\int_{\R^{n+1}_+} t\, \partial_{t} \brac{\lap_{x,t} (A\, B)\, \partial_{t} \Phi} \right |.
\end{split}
\]
Since $A$ and $B$ are harmonic, $\lap_{x,t} (A\, B) = 2 \nabla_{\R^{n+1}} A \cdot \nabla_{\R^{n+1}} B$, and thus
\[
\begin{split}
=& 2\left |\int_{\R^{n+1}_+} t\, \partial_{t} \brac{\nabla_{\R^{n+1}}A\cdot \nabla_{\R^{n+1}}B\, \partial_{t} \Phi}\right |\\
\leq &2\left |\int_{\R^{n+1}_+} t\, \partial_{t} (\nabla_{\R^{n+1}}A\cdot \nabla_{\R^{n+1}}B)\, \partial_{t} \Phi\right | + 2\left |\int_{\R^{n+1}_+} t\, \nabla_{\R^{n+1}}A\cdot \nabla_{\R^{n+1}}B\, \partial_{tt} \Phi \right |.
\end{split}
\]
Again replacing $\partial_{tt} \Phi$ by $-\lap_x \Phi$ and using integration by parts in $x$ for the second term, we arrive at
\[
 =2\left |\int_{\R^{n+1}_+} t\, \partial_{t} (\nabla_{\R^{n+1}}A\cdot \nabla_{\R^{n+1}}B)\, \partial_{t} \Phi \right |+2\left |\int_{\R^{n+1}_+} t\, \nabla_x \brac{\nabla_{\R^{n+1}}A\cdot \nabla_{\R^{n+1}}B}\, \cdot \nabla_x \Phi \right |.
\]
This proves \eqref{eq:hopeful}.
\end{proof}

\begin{proof}[Proof for $s<1$]
Assume that $s < 1$. Set $\tilde{\varphi} := \laps{s} \varphi$, and let $A$, $B$, $\tilde{\Phi}$ be the $s$-harmonic extensions of $a$, $b$, $\tilde{\varphi}$, respectively. That is \[A(x,t) = P_t^s a(x),\ B(x,t) = P_t^s b(x),\ \tilde{\Phi}(x,t) = P_t^s \tilde{\varphi}(x),\] where $P_t^s$ is the Caffarelli-Silvestre Poisson operator as in \eqref{eq:spoissonpde}. 

This time we aim for the following estimate:
\begin{equation}\label{eq:hopefuls}
\begin{split}
 \mathcal{C} \aleq& \int_{\R^{n+1}_+} t^{1-s} |\tilde{\Phi}|\, |\nabla_{x} A|\, |\nabla_{x} B|\\
 &+ \int_{\R^{n+1}_+} t^{2-s} |\tilde{\Phi}|\, \brac{|\nabla_{x} \nabla_{\R^{n+1}} A|\, |\nabla_{\R^{n+1}} B| + |\nabla_{\R^{n+1}} A|\, |\nabla_{x} \nabla_{\R^{n+1}} B|}  \\
 &+ \int_{\R^{n+1}_+} t^{3-s} |\nabla_{\R^{n+1}} \tilde{\Phi}|\, \brac{|\nabla_{\R^{n+1}} \nabla_x A|\, |\nabla_{\R^{n+1}} B| + |\nabla_{\R^{n+1}} A|\, |\nabla_{\R^{n+1}} \nabla_x B|}
 \end{split}
\end{equation}
Observe that $\tilde{\Phi} = P_t \laps{s} \varphi$, thus Proposition~\ref{pr:BMOusualestimate} applied to \eqref{eq:hopefuls} implies \eqref{eq:Hsbmoest} for $s < 1$.

It remains to establish \eqref{eq:hopefuls}. We use integration-by-parts in $t$ and the representation of $\laps{s} a = c\lim_{t \to 0} t^{1-s} \partial_t A$ from \eqref{eq:spoissonpde},
\[
\begin{split}
\mathcal{C} :=&\left |\int_{\R^n} \brac{\laps{s} (ab) - a \laps{s} b - \laps{s} a\ b} \tilde{\varphi} \right |\\
=&\left |\int_{\R^{n+1}_+}  \partial_{t} \brac{t^{1-s}  A\, B\, \partial_{t} \tilde{\Phi} - t^{1-s} A\,  \partial_t B\, \tilde{\Phi}- t^{1-s} \partial_t A\, B\, \tilde{\Phi}}\right |\\
=&\left |\int_{\R^{n+1}_+}  \partial_{t} \brac{t^{1-s}  (A\, B)\, \partial_{t} \tilde{\Phi} - t^{1-s} \partial_t (AB) \tilde{\Phi}} \right |.
\end{split}
\]
Again we use the product rule for $\partial_t$ and have a cancellation
\[
 =\left |\int_{\R^{n+1}_+}  (A\, B)\, \partial_{t} \brac{t^{1-s} \partial_{t} \tilde{\Phi}} - \partial_{t} \brac{t^{1-s} \partial_t (AB)}\, \tilde{\Phi}\right |.
\]
Since $\partial_{t} \brac{t^{1-s} \partial_{t} \tilde{\Phi}} = -t^{1-s}\lap_x \tilde{\Phi}$ and with an integration by parts in $x$,
\[
 =\left |\int_{\R^{n+1}_+}  L_s(A\, B)\, \tilde{\Phi}\right |,
\]
where we set
\[
  L_s (AB) := t^{1-s}\lap_x (A\, B)\, + \partial_{t} \brac{t^{1-s} \partial_t (AB)}.
\]
By \eqref{eq:spoissonpde}, $L_s(A) = L_s (B) = 0$. On the other hand, we have the product rule
\[
  L_s (A\, B)-L_s(A)\, B-A\, L_s(B) = 2 t^{1-s}\, \nabla_{\R^{n+1}} A \cdot \nabla_{\R^{n+1}} B. 
\]
Consequently,
\[
 \mathcal{C} \leq  2\left |\int_{\R^{n+1}_+} t^{1-s} \nabla_{x} A \cdot \nabla_{x} B\, \tilde{\Phi}\right |+2\left |\int_{\R^{n+1}_+} t^{1-s} \partial_t A \, \partial_t B\, \tilde{\Phi} \right |.
\]
The first term is already of the form in \eqref{eq:hopefuls}. As for the second term, we use the following integration by parts formula in $t$-direction
\[
 \int_{0}^{\infty} f(t)\ dt = -\frac{1}{s} \int_0^\infty t^{s} \partial_t \brac{t^{1-s} f(t)}\ dt.
\]
Thus,
\[
\begin{split}
 &\int_{\R^{n+1}_+} t^{1-s} \partial_t A \, \partial_t B\, \tilde{\Phi}\\
 =&\frac{1}{s}\int_{\R^{n+1}_+} t^s \partial_{t} \brac{t^{1-s} \partial_t A \ t^{1-s} \partial_t B\ \tilde{\Phi}}\\
 =&-\frac{1}{s}\int_{\R^{n+1}_+} t^{2-s} \brac{\lap_x A\ \partial_t B\ \tilde{\Phi}} -\int_{\R^{n+1}_+} t^{2-s} \brac{\partial_t A\ \lap_x  B\ \tilde{\Phi}}\\
 &+\frac{1}{s}\int_{\R^{n+1}_+} t^{2-s}  \partial_t A \  \partial_t B\ \partial_{t} \tilde{\Phi}\\
 =&-\frac{1}{s}\int_{\R^{n+1}_+} t^{2-s} \brac{\lap_x A\ \partial_t B\ \tilde{\Phi}} -\int_{\R^{n+1}_+} t^{2-s} \brac{\partial_t A\ \lap_x  B\ \tilde{\Phi}}\\
 &+\frac{1}{2s^2} \int_{\R^{n+1}_+} t^{2s} \partial_{t} \brac{t^{1-s}  \partial_t A \  t^{1-s}\partial_t B\ t^{1-s} \partial_{t} \tilde{\Phi}}\\
\end{split}
 \]
Now we finish by integrating by parts if $\partial_t$ hits $t^{1-s} \partial_{t} \tilde{\Phi}$.
\end{proof}

\section{\texorpdfstring{$L^1$}{L1}-estimate for a double-commutator}\label{s:new}
The Coifman-Rochberg-Weiss theorem, Theorem~\ref{th:CRW}, fails on $L^1$. More generally, it seems that there is no reason that an $L^1$-analogon of the Riesz-transform estimates in \eqref{eq:rieszgf} holds. As an application of our techniques, we show here a replacement estimate that estimates the commutators of those commutators in $L^1$. Denote with $\Hz$ the Hilbert transform (i.e. the one-dimensional Riesz transform $\Rz_1$).
\begin{theorem}\label{th:newtheorem}
For $s_1, s_2 \in (0,1)$ and $s_1 + s_2 = 1$ and any $p \in (1,\infty)$, $q \in [1,\infty]$ we have for any $f,g \in C_c^\infty(\R^n)$,
\begin{equation}\label{eq:new1}
 \left \|[f,\Hz](\laph g)-[g,\Hz](\laph f) \right \|_{L^1(\R)} \aleq \|\laps{s_1} f\|_{L^{(p,q)}(\R)}\ \|\laps{s_2} g\|_{L^{(p',q')}(\R)}.
 \end{equation}
and
\begin{equation}\label{eq:new2}
\left \|\mathcal{H} \brac{[f,\Hz](\laph g)+[g,\Hz](\laph f)} \right \|_{L^1(\R)} \aleq \|\laps{s_1} f\|_{L^{(p,q)}(\R)}\ \|\laps{s_2} g\|_{L^{(p',q')}(\R)}.
\end{equation}
\end{theorem}
\begin{proof}[Proof of \eqref{eq:new1}]
Let
\[
\begin{split}
\mathcal{C} :=& \left |\int_{\R} \brac{[f,\Hz]\laph g-[g,\Hz]\laph f} \varphi \right |\\
 =&\left |\int_{\R} f\ \Hz \laph g\ \varphi + f\ \laph g\ \Hz \varphi - \Hz \laph f\ g\ \varphi - \laph f\ g\ \Hz \varphi \right |.
\end{split}
\]
For the theorem to be proven, by duality, it suffices to show
\begin{equation}\label{eq:newgoal}
 \mathcal{C} \aleq \|\laps{s_1} f\|_{L^{(p,q)}(\R)}\ \|\laps{s_2} g\|_{L^{(p',q')}(\R)}\ \|\varphi\|_{L^\infty(\R)}.
\end{equation}
Let $F := P_t^1 f$, $G := P_t^1 g$, $\Phi := P_t^1 \varphi$ be the respective harmonic extensions. Then, as above, via integration by parts in $t$,
\[
 \mathcal{C} \aleq \left |\int_{\R^2_+} \partial_t \brac{F\, \tilde{\Hz}G_t\, \Phi + F\,  G_t\,  \tilde{\Hz} \Phi - \tilde{\Hz} F_t\, G\, \Phi- F_t\, G\, \tilde{\Hz} \Phi} \right |.
\]
Recall the rules for derivatives of the harmonic extensions of Hilbert transforms:
\begin{equation}\label{eq:hilbertrules}
 \tilde{\Hz} F_t = -F_x, \quad F_t = \tilde{\Hz}F_x.
\end{equation}
Then
\[
\mathcal{C} =\left |\int_{\R^2_+} \partial_t \brac{-F\, G_x\, \Phi + F\, G_t\, \tilde{\Hz} \Phi + F_x\, G\, \Phi- F_t\, G\, \tilde{\Hz} \Phi} \right |.
\]
We compute
\[
\begin{split}
 \mathcal{I} := &\partial_t \brac{-F\, G_x\, \Phi + F\, G_t\, \tilde{\Hz} \Phi + F_x\, G\, \Phi- F_t\, G\, \tilde{\Hz} \Phi}\\
=& \partial_t \brac{-F\, G_x\, \Phi  + F_x\, G\, \Phi}+ \partial_t \brac{F\, G_t\, \tilde{\Hz} \Phi- F_t\, G\, \tilde{\Hz} \Phi}\\
\end{split}
\]
and with a cancellation in the second term,
\[
\begin{split}
=& \brac{-F_t\, G_x\, \Phi  + F_x\, G_t\, \Phi} + \brac{-F\, G_{xt}\, \Phi  + F_{xt}\, G\, \Phi}+ \brac{-F\, G_x\, \Phi_t  + F_x\, G\, \Phi_t}\\
&+ \brac{F\, G_{tt}\, \tilde{\Hz} \Phi- F_{tt}\, G\, \tilde{\Hz} \Phi} + \brac{F\, G_t\, \tilde{\Hz} \Phi_t- F_t\, G\, \tilde{\Hz} \Phi_t}
 \end{split}
 \]
We use \eqref{eq:hilbertrules}, the fact that $\partial_{tt} F = - \partial_{xx} F$,
\begin{equation}\label{eq:new:s1}
\begin{split}
=& \brac{-F_t\, G_x\, \Phi  + F_x\, G_t\, \Phi} + \brac{-F\, G_{xt}\, \Phi  + F_{xt}\, G\, \Phi}+ \brac{-F\, G_x\, \Phi_t  + F_x\, G\, \Phi_t}\\
&+ \brac{-F\, G_{xx}\, \tilde{\Hz} \Phi+ F_{xx}\, G\, \tilde{\Hz} \Phi} + \brac{-F\, G_t\, \Phi_x+ F_t\, G\, \Phi_x}
 \end{split}
 \end{equation}
Next, again with the help of \eqref{eq:hilbertrules},
\[
 -F\, G_{xx}\, \tilde{\Hz} \Phi = -(F\, G_{x}\, \tilde{\Hz} \Phi)_x + F_x\, G_x\, \tilde{\Hz} \Phi+ F\, G_{x}\, \Phi_t
\]
\[
 F_{xx}\, G\, \tilde{\Hz} \Phi = (F_{x}\, G\, \tilde{\Hz} \Phi)_x - F_{x}\, G_x\, \tilde{\Hz} \Phi - F_{x}\, G\, \Phi_t
\]
Plugging this into \eqref{eq:new:s1}, more terms cancel,
\begin{equation}\label{eq:new:s2}
\begin{split}
\mathcal{I}=& \brac{-F_t\, G_x\, \Phi  + F_x\, G_t\, \Phi} + \brac{-F\, G_{xt}\, \Phi  + F_{xt}\, G\, \Phi}\\
&+ \brac{
(F_{x}\, G\, \tilde{\Hz} \Phi)_x-(F\, G_{x}\, \tilde{\Hz} \Phi)_x 
} + \brac{-F\, G_t\, \Phi_x+ F_t\, G\, \Phi_x}.
 \end{split}
 \end{equation}
We repeat this strategy with
\[
 -F\, G_{xt}\, \Phi  = -(F\, G_{t}\, \Phi)_x+F_x\, G_{t}\, \Phi+F\, G_{t}\, \Phi_x,
\]
\[
 F_{xt}\, G\, \Phi = (F_{t}\, G\, \Phi)_x - F_{t}\, G_x\, \Phi - F_{t}\, G\, \Phi_x.
\]
This we plug into \eqref{eq:new:s2}, and arrive at
\[
\begin{split}
\mathcal{I}=& 2\brac{F_x\, G_t\, \Phi-F_t\, G_x\, \Phi} + \brac{F_{t}\, G\, \Phi-F\, G_{t}\, \Phi+
F_{x}\, G\, \tilde{\Hz} \Phi-F\, G_{x}\, \tilde{\Hz} \Phi 
}_x.
 \end{split}
 \]
The second term vanishes when integrating in $x$, and thus
\begin{equation}\label{eq:newjacobianform}
 \mathcal{C} = \left |\int_{\R^2_+}  \mathcal{I} \right | = 2\left |\int_{\R^2_+} \det(\nabla_{\R^2} F, \nabla_{\R^2} G)\ \Phi \right |.
\end{equation}
With Proposition~\ref{pr:Lpqest} we obtain \eqref{eq:newgoal}.
\end{proof}

\begin{remark}
In \eqref{eq:newjacobianform}, having the determinant structure, one might hope to use the Hardy-BMO duality, in form of Theorem~\ref{th:CLMSjac}, to obtain (in view of Proposition~\ref{pr:BMOextension}) an estimate only in terms of $[\varphi]_{BMO}$ instead of $\|\varphi\|_{L^\infty}$. If that was the case, then in Theorem~\ref{th:newtheorem} we had a Hardy-space bound instead of merely the $L^1$ bound. However, we were not able to do this, the reason being that the integral is on the half-space and for a reflection argument we would need to estimate $\frac{t}{|t|} \Phi(|t|,x)$ in $BMO$. However, even though $\Phi(|t|,x)$ is in $BMO$, see Proposition~\ref{pr:BMOextension}, there is no reason $\frac{t}{|t|} \Phi(|t|,x)$ belongs to $BMO$ as well.
\end{remark}

\begin{proof}[Proof of \eqref{eq:new2}]
As in the proof of \eqref{eq:new1} let $F := P_t^1 f$, $G := P_t^1 g$, $\Phi := P_t^1 \varphi$ be the respective harmonic extensions. Then,
\[
\begin{split}
\mathcal{C} :=& \left | \int_{\R} \brac{[f,\Hz](\laph g)+[g,\Hz](\laph f)}\, \mathcal{H} \varphi \right |\\
=&\left |\int_{\R^2_+} \partial_t \brac{F\, \tilde{\Hz}G_t\, \tilde{\Hz}\Phi - F\, G_t\ \Phi + \tilde{\Hz}F_t\,G\,  \tilde{\Hz}\Phi - F_t\, G\, \Phi} \right |.
\end{split}
\]
This time we compute,
\[
 \begin{split}
\mathcal{I} := &\partial_t \brac{F\, \tilde{\Hz}G_t\, \tilde{\Hz}\Phi - F\, G_t\ \Phi + \tilde{\Hz}F_t\,G\,  \tilde{\Hz}\Phi - F_t\, G\, \Phi} \\
=&\partial_t \brac{-(F\, G)_x\, \tilde{\Hz}\Phi } + \partial_t \brac{ - (F\, G)_t\ \Phi} \\
=&-\brac{F\, G\, \tilde{\Hz}\Phi }_{xt} + \partial_t \brac{(F\, G)\, \Phi_t  - (F\, G)_t\ \Phi} \\
=&-\brac{F\, G\, \tilde{\Hz}\Phi }_{xt} + (F\, G)\, \Phi_{tt}  - (F\, G)_{tt}\ \Phi \\
=&-\brac{F\, G\, \tilde{\Hz}\Phi }_{xt} - (F\, G)\, \Phi_{xx}  - (F\, G)_{tt}\ \Phi \\  
 \end{split}
\]
Thus, since $(FG)_{xx} + (FG)_{tt} = 2\nabla_{\R^2} F \cdot \nabla_{\R^2} G$, we have
\begin{equation}\label{eq:new2:result}
 \mathcal{C} = 2\left |\int_{\R^2_+} \nabla_{\R^2} F\cdot \nabla_{\R^2} G\ \Phi\right |.
\end{equation}
We conclude as we did for \eqref{eq:new1}.
\end{proof}
\begin{remark}
Actually, our computations show that the left-hand sides in \eqref{eq:new1} and \eqref{eq:new2} are essentially the same estimates, more precisely,
\[
  \left \|[f,\Hz](\laph g)-[g,\Hz](\laph f) \right \|_{L^1(\R)} = \left \|\mathcal{H} \brac{[f,\Hz](\Hz\laph g)+[\Hz g,\Hz](\laph f)} \right \|_{L^1(\R)}.
\]
To see this, replace in \eqref{eq:new2:result} $g$ with $\Hz g$, that is $\tilde{G} := \tilde{\Hz}G$. Then in view of \eqref{eq:hilbertrules}, the equation \eqref{eq:new2:result} becomes \eqref{eq:newjacobianform},
\[
 2\left |\int_{\R^2_+} \nabla_{\R^2} F\cdot \nabla_{\R^2} G\ \Phi\right | = 2\left |\int_{\R^2_+} \det(\nabla_{\R^2} F,\nabla_{\R^2} \tilde{G})\ \Phi\right |.
\]

\end{remark}

\section{Trace theorems}\label{s:poissonchar}
Characterizations of function spaces via the Poisson potential have a long tradition, see in particular Stein's books \cite{Stein-Singular-Integrals, Stein-Harmonic-Analysis}. In this section we gather such characterizations for the Poisson operator. We postpone references to literature and proofs to Appendix~\ref{s:proofpoisson}. 

Recall the definition of the Poisson extension operator from Section~\ref{s:poissonoperator}.
Let $F^s(x,t) = P_t^s f(x)$ for some $s \in (0,2)$ and $f \in C_c^\infty(\R^n)$. With $\mathcal{M} f$ we denote the Hardy-Littlewood maximal function
\[
 \mathcal{M} f(x) = \sup_{B \ni x} |B|^{-1} \int_B |f|\ dz
\]
where the supremum is over balls $B$ containing $x$.

\begin{proposition}[Pointwise estimates]\label{pr:pc:maximal}
We have for any $k \in \N_0$,
\begin{equation}\label{eq:tinfty}
\sup_{(x,t) \in \R^{n+1}_+} t^{n+k} |\nabla^k_{\R^{n+1}} F^s(x,t)| \leq C_s\, \|f\|_{L^1(\R^n)},
\end{equation}
\begin{equation}\label{eq:tinftyLinfty}
\sup_{(x,t) \in \R^{n+1}_+} t^{k}|\nabla^k_{\R^{n+1}} F^s(x,t)| \leq C_s\, \|f\|_{L^\infty(\R^n)}.
\end{equation}
Also, we have the following estimates in terms of the maximal function,
\begin{equation}\label{eq:possionmaximal}
 \sup_{(y,t):\, |y-x| < t} |F^s(y,t)| \leq C_s\, \mathcal{M}f(x).
\end{equation}
For any $s\leq 1$,
\begin{equation}\label{eq:possionmaximalpt}
 \sup_{(y,t):\, |y-x| < t} |t^{1-s}\partial_t F^s(y,t)| \leq C_s\, \mathcal{M}(\laps{s} f)(x).
\end{equation}
Finally, denoting $\nabla^s = \nabla \lapms{1-s} = \Rz \laps{s}$ what is sometimes called a ``fractional gradient'' (i.e. the vectorial Riesz transform $\Rz$ applied the $\laps{s}$), we have
\begin{equation}\label{eq:possionmaximalptx}
 \sup_{(y,t):\, |y-x| < t} |t^{1-s}\nabla_{x} F^s(y,t)| \leq C_s\, \mathcal{M}(\nabla^{s} f)(x).
\end{equation}
Finally, for any $\sigma > 0$,
\begin{equation}\label{eq:maximalpotential}
\sup_{t > 0} t^{\sigma} |P_t^s f|(x) \aleq \mathcal{M}(\lapms{\sigma} f)(x),
\end{equation}
where $\lapms{\sigma}$ is the Riesz potential.
\end{proposition}

As usual, the norm of the Lebesgue-spaces $L^p$ is defined as
\[
 \|f\|_{L^p(\R^n)} = \brac{\int_{\R^n} |f|^p}^{\frac{1}{p}} \quad p \in [1,\infty),
\]
\[
 \|f\|_{L^\infty(\R^n)} = \esssup_{x \in \R^n} |f(x)|.
\]
A finer scale than Lebesgue-spaces are the Lorentz-spaces $L^{(p,q)}$, $q \in [1,\infty]$ -- see, e.g., \cite{Hunt-1966, Tartar-2007, GrafakosCF}. For $q = p$ they are the same as Lebesgue spaces, $L^{(p,p)} = L^p$. They are defined as follows. 
For measurable functions $f: \R^n \rightarrow \R$ the \emph{decreasing rearrangement} $f^\ast(t)$, $t  > 0$, is defined as
\[
 f^\ast (t):= \inf\left\{s>0: {\mathcal L}^n(\{|f|> s\})\leq t\right\}.
\]
Here, $\mathcal L^n$ denotes the Lebesgue measure. The Lorentz-space norm $\|\cdot \|_{L^{(p,q)}}$ is given by
\begin{equation*}
 \|f\|_{L^{(p,q)}} := \begin{cases}
                      \left(\int\limits_0^\infty (t^{1/p} f^\ast (t)) ^q \frac {dt}{t}\right)^{1/q}  
			&\text{if } p,q \in[1, \infty),\\
                      \sup\limits _{t>0} t^{1/p} f^\ast (t)&\text{if } q=\infty.
                     \end{cases}
\end{equation*}

The fractional Sobolev spaces $W^{\nu,p}$, $\nu \in (0,1)$, have the seminorm
\[
 [f]_{\dot{W}^{\nu,p}(\R^n)} = \brac{\int_{\R^n}\int_{\R^n}  \frac{|f(x)-f(y)|^p}{|x-y|^{n+\nu p}}\ dx\, dy}^{\frac{1}{p}}.
\]
If $p \neq 2$, another fractional Sobolev space, sometimes denoted with $\dot{H}^{\nu,p}$ is defined via the seminorm $\|\laps{\nu} f\|_{L^p(\R^n)}$.

We turn to characterizations for Sobolev spaces. In the following, it is crucial to observe the different orders up to which the characterization holds. 
The general rule is: the order of the derivative on the extension $F^s(x,t)$ has to be strictly larger than the order of the Sobolev space we want to characterize. 
However, and this is very important to observe when $s \neq 1$, the $t$-direction derivatives $t^{1-s} \partial_t$ count only as being ``of order $s$''. This is by construction of the Poisson potentials $P_t^s$: they are supposed to satisfy $\lim_{t\to 0} t^{1-s} \partial_t F^s(x,t) = c \laps{s} f$.

\begin{proposition}[Fractional Sobolev spaces]\label{pr:Sobolevspaces}
The following holds whenever $p \in (1,\infty)$, $q \in [1,\infty]$.

For $s \in (0,2)$, $\nu \in (0,1)$,
\begin{equation}\label{eq:Wspequivalencenablax}
 \brac{\int_{\R^n} \int_{0}^\infty |t^{1-\frac{1}{p}-\nu} \nabla_{x} F^s(x,t)|^p\ dt\ dx}^{\frac{1}{p}} \aeq [f]_{\dot{W}^{\nu,p}(\R^n)}.
\end{equation}
For $s \in (0,2)$, $\nu \in (0,2)$,
\begin{equation}\label{eq:Wspequivalencenablaxx}
 \brac{\int_{\R^n} \int_{0}^\infty |t^{2-\frac{1}{p}-\nu} \nabla^2_{x} F^s(x,t)|^p\ dt\ dx}^{\frac{1}{p}} \aeq [f]_{\dot{W}^{\nu,p}(\R^n)}.
\end{equation}
For $s \in (0,2)$, $\nu \in (0,1)$, $\nu < s$, 
\begin{equation}\label{eq:Wspequivalencedt}
 \brac{\int_{\R^n} \int_{0}^\infty |t^{1-\frac{1}{p}-\nu} \partial_t F^s(x,t)|^p\ dt\ dx}^{\frac{1}{p}} \aeq [f]_{\dot{W}^{\nu,p}(\R^n)}.
\end{equation}

For $s \in (0,2)$, $\nu \in [0,1)$,
\begin{equation}\label{eq:Hspequivalencenablax}
 \left \|x \mapsto \brac{\int_{0}^\infty |t^{\frac{1}{2}-\nu} \nabla_{x} F^s(x,t)|^2\ dt}^{\frac{1}{2}} \right \|_{L^{(p,q)}(\R^n)} \aeq \|\laps{\nu} f\|_{L^{(p,q)}(\R^n)}.
\end{equation}
For $s \in (0,2)$, $\nu \in [0,1)$, $\nu < s$ 
\begin{equation}\label{eq:Hspequivalencedt}
 \left \|x \mapsto  \brac{\int_{0}^\infty |t^{\frac{1}{2}-\nu}\partial_t F^s(x,t)|^2\ dt}^{\frac{1}{2}} \right \|_{L^{(p,q)}(\R^n)} \aeq \|\laps{\nu} f\|_{L^{(p,q)}(\R^n)}.
\end{equation}
For $s \in (0,2)$, $\nu \in [0,2)$,
\begin{equation}\label{eq:Hspequivalencenablaxx}
 \left \|x \mapsto \int_{\R^n} \brac{\int_{0}^\infty |t^{\frac{3}{2}-\nu} \nabla^2_{x} F^s(x,t)|^2\ dt}^{\frac{1}{2}} \right \|_{L^{(p,q)}(\R^n)} \aeq \|\laps{\nu} f\|_{L^{(p,q)}(\R^n)}.
\end{equation}
The estimates \eqref{eq:Hspequivalencenablax}, \eqref{eq:Hspequivalencedt}, \eqref{eq:Hspequivalencenablaxx} also hold for $\nu < 0$ with $\laps{\nu} f$ replaced by the Riesz potential $\lapms{|\nu|} f$.
\end{proposition}

We also record the following characterizations in terms of so-called nontangential square functions.
\begin{proposition}[Square function estimates]\label{pr:Lorentzspaces}
For $1 < p < \infty$, $q \in [1,\infty]$ and $s \in (0,1]$, and any $\nu \in [0,s)$,
\begin{equation}\label{eq:squarefctchar}
\left \|x \mapsto \left (\int_{(y,t): |y-x|<t} t^{1-2\nu-n}| \partial_t F^s(y,t) |^2\ dy\, dt \right )^{\frac{1}{2}} \right \|_{L^{(p,q)}(\R^n)} \aeq \|\laps{\nu} f\|_{L^{(p,q)}(\R^n)}.
\end{equation}
If we replace $|\partial_t F^s(y,t)|$ by $|\nabla_x F^s(y,t)|$ this estimate holds for any $\nu \in [0,1)$.

Moreover, for $s \in (0,1]$, $\nu \in (0,1+s)$,
\begin{equation}\label{eq:squarefctnablachar}
\left \|x \mapsto \left (\int_{(y,t): |y-x|<t} t^{3-2\nu-n}|\nabla_x \nabla_{\R^{n+1}} F^s(y,t) |^2\ dy\, dt \right )^{\frac{1}{2}} \right \|_{L^{(p,q)}(\R^n)} \aeq \|\laps{\nu} f\|_{L^{(p,q)}(\R^n)}
\end{equation}
\end{proposition}

Next, we consider BMO. Denoting $(f)_B \equiv \mvint_{B} f:= |B|^{-1} \int_B f$ the mean value integral on $B$, the BMO-seminorm is given by
\begin{equation}\label{eq:BMOseminorm}
 [f]_{BMO} = \sup_{B \subset \R^n} \mvint_{B} \abs{f-(f)_B},
\end{equation}
where the supremum is over balls $B \subset \R^n$. There is a well-known relation between BMO and certain Carleson-measures on $\R^{n+1}_+$. This takes the following form.
\begin{proposition}[BMO-Characterization]\label{pr:BMOcharacterization}
For $s \in (0,2)$,
\begin{equation}\label{eq:BMOequivalence}
 [f]_{BMO(\R^n)} \aeq  \sup_{B \subset \R^n} \brac{|B|^{-1}\int_{T(B)}  t|\nabla_{\R^{n+1}} F^s(y,t)|^2\, dy\, dt}^{\frac{1}{2}}.
\end{equation}
Here the supremum is over balls $B \subset \R^n$ and $T(B)$ is the ``tent'' over $B$ in $\R^n$, i.e. $T(B_r(x_0)) = \{(x,t) \in \R^{n+1}_+: |x-x_0| < r-t\}$.
\end{proposition}

As an interesting observation, we also state the following result which treats even reflection of the harmonic extension.
\begin{proposition}\label{pr:BMOextension}
Let $f \in C_c^{\infty}(\R^n)$. We set $F^e$ to be the harmonic extension to $\R^{n+1}_+$ evenly reflected to a function on $\R^{n+1}$. That is,
\[
 F^e(x,t) := P^1_{|t|} f(x).
\]
Then we have
\[
 [F^e]_{BMO(\R^{n+1})} \aleq \|f\|_{BMO(\R^n)}.
\]
\end{proposition}

We turn to H\"older- and Lipschitz spaces. We denote the H\"older semi-norm by, $\nu > 0$
\[
 [f]_{C^{\nu}(\R^n)} := \begin{cases} 
                         \sup\limits_{x \neq y \in \R^n} \frac{|\nabla^{\lfloor \nu \rfloor} f(x)-\nabla^{\lfloor \nu \rfloor}f(y)|}{|x-y|^{\nu-\lfloor \nu\rfloor}} \quad &\mbox{if $\nu \not \in \N$}\\
                         \|\nabla^\nu f\|_{L^\infty} \quad &\mbox{if $\nu \in \N$}\\
                        \end{cases}
\]
As usual we will denote $[f]_{\lip} = [f]_{C^{1}}$.

\begin{proposition}[H\"older spaces]\label{pr:hoelder}
For $\nu \in (0,s)$,
\begin{equation}\label{eq:Hoelderequivalence}
 \sup_{(x, t) \in \R^{n+1}_+ } t^{1-\nu}|\partial_t F^s(x,t)| \aeq [f]_{C^{\nu}(\R^n)}.
\end{equation}
and for any $\nu \in (0,1]$,
\begin{equation}\label{eq:Hoelderequivalencelimit}
 \sup_{(x, t) \in \R^{n+1}_+ } t^{1-\nu}|\nabla_{x} F^s(x,t)| \aleq [f]_{C^{\nu}(\R^n)}.
\end{equation}
\end{proposition}

For sake of completeness, let us also mention that one can characterize the full range of Besov- and Triebel-Lizorkin spaces in terms of Poisson-type potentials $P^s_t \laps{\beta} f$ and $P^s_t \nabla f$. This follows from a general characterization of those spaces via convolution operators, recently given by Bui and Candy in \cite{BuiCandy-2015}. Again for $s \neq 1$, the maximal differential order of the spaces which can be characterized is different depending on whether we use $\partial_t F^s(x,t)$ or $\nabla_x F^s(x,t)$.

\begin{theorem}\label{th:BTcharacterization}
For any $\beta > \alpha$, $\beta > 0$, $0 < p,q <\infty$, $s > 0$,
\[
 \|f\|_{\dot{B}^{\alpha}_{p,q}} \aeq \brac{\int_0^\infty \brac{\int_{\R^n} |t^{-\frac{1}{q}-\alpha+\beta}  P^s_t \laps{\beta} f(x)|^p\ dx }^{\frac{q}{p}} dt}^{\frac{1}{q}}\\
\]
\[
 \|f\|_{\dot{F}^{\alpha}_{p,q}} \aeq \brac{\int_{\R^n}\brac{\int_0^\infty  |t^{-\frac{1}{q}-\alpha+\beta}  P^s_t \laps{\beta} f(x)|^q\ dt }^{\frac{p}{q}} dx}^{\frac{1}{p}}.
\]
If $\alpha < 1$,
\[
 \|f\|_{\dot{B}^{\alpha}_{p,q}} \aeq \brac{\int_0^\infty \brac{\int_{\R^n} |t^{-\frac{1}{q}-\alpha+1}  \nabla_x P^s_t f(x)|^p\ dx }^{\frac{q}{p}} dt}^{\frac{1}{q}}\\
\]
\[
 \|f\|_{\dot{F}^{\alpha}_{p,q}} \aeq \brac{\int_{\R^n}\brac{\int_0^\infty  |t^{-\frac{1}{q}-\alpha+1}  \nabla_x P^s_t f(x)|^q\ dt }^{\frac{p}{q}} dx}^{\frac{1}{p}}.
\]
Regarding derivatives in $t$, we have the following for $\alpha < s$,
\[
 \|f\|_{\dot{B}^{\alpha}_{p,q}} \aeq \brac{\int_0^\infty \brac{\int_{\R^n} |t^{-\frac{1}{q}-\alpha+1}  \partial_t P^s_t f(x)|^p\ dx }^{\frac{q}{p}} dt}^{\frac{1}{q}},
\]
\[
 \|f\|_{\dot{F}^{\alpha}_{p,q}} \aeq \brac{\int_{\R^n}\brac{\int_0^\infty  |t^{-\frac{1}{q}-\alpha+1}  \partial_t P^s_t f(x)|^q\ dt }^{\frac{p}{q}} dx}^{\frac{1}{p}}.
\]
Regarding two derivatives, for $\alpha < s+1$,
\[
 \|f\|_{\dot{B}^{\alpha}_{p,q}} \aeq \brac{\int_0^\infty \brac{\int_{\R^n} |t^{-\frac{1}{q}-\alpha+2}  \partial_t \nabla_{x} P^s_t f(x)|^p\ dx }^{\frac{q}{p}} dt}^{\frac{1}{q}},
\]
\[
 \|f\|_{\dot{F}^{\alpha}_{p,q}} \aeq \brac{\int_{\R^n}\brac{\int_0^\infty  |t^{-\frac{1}{q}-\alpha+2}  \partial_t \nabla_x P^s_t f(x)|^q\ dt }^{\frac{p}{q}} dx}^{\frac{1}{p}}.
\]
\end{theorem}

\section{Useful blackbox estimates from \texorpdfstring{$\R^{n+1}_+$}{R(n+1)} to \texorpdfstring{$\R^n$}{R(n)}}\label{s:bbestimates}
As a consequence of Section~\ref{s:poissonchar} we obtain the following estimates. The proofs can be found in Appendix~\ref{s:proofbbestimates}. Recall that $\lapms{\sigma}$ denotes the Riesz potential, the inverse of $\laps{\sigma}$.

First we consider estimates with $L^p$-spaces.
\begin{proposition}[$L^p$-estimates]\label{pr:Lpqest}
Let $s \in (0,1]$. Take any $p_i \in (1,\infty)$ and $q_i \in [1,\infty]$, for $i \in \{1,2,3\}$, and such that 
\[
\sum_{i =1}^3 \frac{1}{p_i}=\sum_{i =1}^3 \frac{1}{q_i} = 1.
\]
Denote with $F^s(x,t) = P_t^s f(x)$, $G^s(x,t) = P_t^s g(x)$, $H^s(x,t) := P_t^s h(x)$ for  $f,g,h \in C_c^\infty(\R^n)$. .
For $s_1, s_2 \in [0,s)$ and $s_3 \geq 0$,
\[
\begin{split}
\int_{\R^{n+1}_+} t^{1-s_1-s_2+s_3}\ |\nabla_{\R^{n+1}} F|\, |\nabla_{\R^{n+1}} G|\, |H|
\aleq  \|\laps{s_1} f\|_{L^{(p_1,q_1)}}\ \|\laps{s_2} g\|_{L^{(p_2,q_2)}}\ \|\lapms{s_3} h\|_{L^{(p_3,q_3)}},
\end{split}
\]
and for $s_1 \in [0,s)$ and $s_2, s_3 \geq 0$,
\[
\begin{split}
\int_{\R^{n+1}_+} t^{1-s_1+s_2+s_3}\ |\nabla_{\R^{n+1}} F|\, |\nabla_{\R^{n+1}} G|\, |H|
\aleq  \|\laps{s_1} f\|_{L^{(p_1,q_1)}}\ \|\lapms{s_2} g\|_{L^{(p_2,q_2)}}\ \|\lapms{s_3} h\|_{L^{(p_3,q_3)}}.
\end{split}
\]
If $\nabla_{\R^{n+1}} F$ is replaced with $\nabla_x F$ then $s_1 \in [0,1)$ is allowed in the above two estimates.

For $s_1 \in (0,1+s)$, and $s_2, s_3 \geq 0$.
\[
\begin{split}
\int_{\R^{n+1}_+} t^{2-s_1-s_2+s_3}\ |\nabla_x \nabla_{\R^{n+1}} F|\, |\nabla_{\R^{n+1}} G|\, |H|
\aleq  \|\laps{s_1} f\|_{L^{(p_1,q_1)}}\ \|\laps{s_2} g\|_{L^{(p_2,q_2)}}\ \|\lapms{s_3} h\|_{L^{(p_3,q_3)}},
\end{split}
\]
If $s_2 < 0$ the last estimate still holds with $\laps{s_2} g$ replaced by $\lapms{|s_2|} g$.

In all terms above we may replace $|H|$ by $t|\nabla_{\R^{n+1}} H|$.

All estimates also hold if $(p_3,q_3) = (\infty,\infty)$.
\end{proposition}

Next, we list estimates involving the BMO-norm.
\begin{proposition}[BMO-estimates]\label{pr:BMOusualestimate}
Let $\ell \geq 1$, $s \in (0,1]$. We have the following estimates for $F^s_i(x,t) = P_t^s f_i(x)$, $G^s(x,t) = P_t^s g(x)$, and $\Phi^s(x,t) = P_t^s \varphi(x)$ for $f_i,g,\varphi \in C_c^\infty(\R^n)$. 

Assume that $p_i \in (1, \infty)$, $q_i \in [1,\infty]$, for $i \in \{0,\ldots,\ell\}$ such that \[\sum_{i=0}^\ell \frac{1}{p_i} = \sum_{i=0}^\ell \frac{1}{q_i} = 1.\] Then
\[
\int_{\R^{n+1}_+} t^{1+(1-s)(\ell+1)} |\nabla_{\R^{n+1}} \Phi^s(x,t)|\, |\nabla_{x} \nabla_{\R^{n+1}} G^s(x,t)|\, |\nabla_{\R^{n+1}} F_1^s(x,t)|\ldots |\nabla_{\R^{n+1}} F_\ell^s(x,t)|\, d(x,t)
\]
\[
\aleq  [\varphi]_{BMO}\, \|\laps{s} g\|_{L^{(p_0,q_0)}}\,  \|\laps{s} f_1\|_{L^{(p_1,q_1)}}\, \ldots\, \|\laps{s} f_\ell\|_{L^{(p_\ell,q_\ell)}}.
\]
Also, for $\nu \in [0,s)$
\[
\int_{\R^{n+1}_+} t^{1-\nu} |\nabla_{\R^{n+1}} \Phi^s(x,t)|\, |\partial_t G^s(x,t)|\, |F_1^s(x,t)|\ldots |F_\ell^s(x,t)|\, d(x,t)
\]
\[
\aleq [\varphi]_{BMO}\, \|\laps{\nu} g\|_{L^{(p_0,q_0)}}\,  \| f_1\|_{L^{(p_1,q_1)}}\, \ldots\, \|f_\ell\|_{L^{(p_\ell,q_\ell)}}.
\]
The last estimate also holds
\begin{itemize}
\item if we replace $|\partial_t G^s(x,t)|$ with $|\nabla_x G^s(x,t)|$ for any $\nu \in [0,1)$.
\item if we replace $|\partial_t G^s(x,t)|$ with $t|\nabla_x \nabla_{\R^{n+1}}G^s(x,t)|$ for any $\nu \in [0,1+s)$.
\item if we replace $|F_1^s(x,t)|$ with $t|\nabla_{\R^{n+1}} F_1^s(x,t)|$ for any $\nu \in [0,s]$.
\end{itemize}
All the above estimate also hold if we replace $|\nabla_{\R^{n+1}} \Phi^s(x,t)|$ with $|t^{s-1}P^s_t (\laps{s} \varphi)|$ or $|t^{s}\nabla_{\R^{n+1}} P^s_t (\laps{s} \varphi)|$.
\end{proposition}

Lastly, we state estimates involving the H\"older-norm.
\begin{proposition}[H\"older-space estimates]\label{pr:Hoelderusualestimate}
Let $\ell \geq 1$, $s \in (0,1]$. We have the following estimates for $F^s(x,t) = P_t^s f(x)$, $G^s(x,t) = P_t^s g(x)$, and $\Phi^s(x,t) = P_t^s \varphi(x)$ for $f,g,\varphi \in C_c^\infty(\R^n)$. Assume that $p_1,p_2 \in (1, \infty)$, $q_1,q_2 \in [1,\infty]$ such that \[\frac{1}{p_1} + \frac{1}{p_2}= \frac{1}{q_1} + \frac{1}{q_2} = 1.\] 

For $\nu \in (0,s)$, $s_1 \in (0,s)$, $s_2 > 0$,
\[
\int_{\R^{n+1}_+} t^{1-\nu-s_1+s_2} |\nabla_{\R^{n+1}} \Phi^s(x,t)|\, |\nabla_{\R^{n+1}} G^s(x,t)|\, |F^s(x,t)|\, d(x,t) 
\]
\[
\aleq  [\varphi]_{C^\nu}\, \|\laps{s_1} g\|_{L^{(p_1,q_1)}}\,  \| \lapms{s_2} f\|_{L^{(p_2,q_2)}}\ 
\]
\begin{itemize}
\item For $\nu = 0$ we replace $[\varphi]_{C^\nu}$ with $\|\varphi\|_{L^\infty}$.
\item For $s_2 = 0$ and $\nu < 1$ we replace $[\varphi]_{C^\nu}$ with $[\laps{\nu} \varphi]_{BMO}$
\item For $s_1 = 0$ we replace $\|\laps{s_1} g\|_{L^{(p_1,q_1)}}$ with $\|g\|_{L^{(p_1,q_1)}}$.
\item If we replace $|\nabla_{\R^{n+1}} \Phi^s(x,t)|$ with $|\nabla_{x} \Phi^s(x,t)|$ we can take $\nu \in (0,1]$.
\item If we replace $ |\nabla_{\R^{n+1}} \Phi^s(x,t)|$ with $t |\nabla_{\R^{n+1}} \nabla_x \Phi^s(x,t)|$, we may take $\nu \in (0,1+s)$.
\item If we take $s_1 < 0$, then $\|\laps{s_1} g\|_{L^{(p_1,q_1)}}$ needs to be replaced with $\|\lapms{|s_1|} g\|_{L^{(p_1,q_1)}}$.
\item We may replace $|F^s(x,t)|$ with $t|\nabla_{\R^{n+1}} F^s(x,t)|$. Then, if $s_2 = 0$ the norm for $\varphi$ is $[\varphi]_{C^\nu}$.
\end{itemize}
For $\nu = 1$ and $s_2 = 0$, observe that 
\[
 [\laps{\nu} \varphi]_{BMO} + [D^{\nu} \varphi]_{BMO} \aleq [\varphi]_{\lip}.
\]
\end{proposition}

\begin{appendix}
\section{Proofs and Literature for Section~\texorpdfstring{\ref{s:poissonchar}}{\ref*{s:poissonchar}}}\label{s:proofpoisson}
There are several versions of trace characterizations of function (harmonically or otherwise) extended to the upper half-space. We are going to use probably the most classical one, the tent spaces and Carleson measures. In \cite{Do-Thiele-2015} they introduce ``outer $L^p$-spaces, which might also offer a way to deal with the traces.
\subsection{The Fourier transform of the Poisson-Bessel potential}\label{s:fouriertransform}
We recall that the (generalized) Poisson potential $P^s_t f$ is given as a convolution operator $P^s_t f = p^s_t \ast f$, where the kernel $p^s_t$ is a Bessel-potential kernel,
\[
 p^s_t(z) := \frac{t^{s}}{\brac{|z|^2 + t^2}^{\frac{n+s}{2}}} = t^{-n} p^s_1(z/t).
\]
A direct computation gives
\begin{equation}\label{eq:poissonkernelest}
 \|p^s_t\|_{L^1(\R^n)} = C \quad \forall t > 0, \quad \|p^s_t\|_{L^\infty(\R^n)} = C\, t^{-n}.
\end{equation}

To apply the characterization for Triebel spaces of Bui and Candy \cite{BuiCandy-2015} one needs to find the growth of the Fourier transform $\mathcal{F}$ of $p^s_t$. 

The case $s =1$ is well known, $\mathcal{F} (p^1_t)(\xi) = e^{-c t|\xi|}$. Indeed, the conditions $(\partial_{tt} + \lap_x) (p^s_t \ast f) = 0$ and $p^s_t \ast f \Big|_{t=0} = f$ are transformed into an ordinary differential equation under the Fourier transform in $x$-variables. Namely, $\sigma(t) := \mathcal{F} (p^1_t)(\xi)$ has to satisfy the equation
\[
 \begin{cases}
  \partial_{tt} \sigma(t) - c |\xi|^2 \sigma(t) = 0 \quad &t \in \R_+\\
  \sigma(0) = 1.
 \end{cases}
\]
In this sense, some authors write $P_t^1 = e^{-ct \sqrt{-\lap}}$.

For $s \neq 1$ this is more involved. Observe that $P_t^s \neq e^{-ct \laps{s}}$. That extension $\tilde{F}(x,t) := e^{-t\laps{s}} f$ is in principle possible as well, is simpler and has the right boundary behavior. But its major, and for our purpose crucial, disadvantage is that the extended objects $\tilde{F}$ do not satisfy a \emph{local} equation, but rather the \emph{nonlocal} equation $(\partial_{tt} + \laps{2s}) \tilde{F} = 0$.

In our case, as introduced by Caffarelli and Silvestre \cite{CaffarelliSilvestre07}, $P_t^s$ is a Bessel potential. The following calculations for $s \neq 1$ can be found, e.g., in \cite[Proposition 7.6]{Hao-HA}. We have
\begin{equation}\label{eq:Fourierpoisson}
 \mathcal{F}(p^s_t)(\xi) = c_s \int_{0}^\infty \lambda^{\frac{s}{2}}\, e^{-\lambda-\frac{|t\xi|^2}{c\lambda }}\, \frac{d\lambda}{\lambda}.
\end{equation}
Here $c > 0$ is a uniform constant, and $c_s$ depends only on dimension and $s$.

In  \cite[Proposition 7.6]{Hao-HA} one can also find the following estimates: for any multiindex $\kappa$,
\[
 |\partial_{\xi^\kappa }\mathcal{F}(p^s_1)(\xi)| \aleq \max\{1,|\xi|^{s-|\kappa|}\} \quad \mbox{for $|\xi| \leq 2$},
\]
\[
 |\partial_{\xi^\kappa }\mathcal{F}(p^s_1)(\xi)| \aleq e^{-c_1 |\xi|} \quad \mbox{for $|\xi| > 2$},
\]

Moreover, setting $q^s_1 := (\partial_t p_t)\big |_{t =1}$,
\[
 |\partial_{\xi^\kappa }\mathcal{F}(q^s_1)(\xi)| \aleq \max\{|\xi|^{s-|\kappa|},1\} \quad \mbox{for $|\xi| \leq 2$}.
\]

\subsection{The pointwise estimates: Proposition~\ref{pr:pc:maximal}}
\begin{proof}[Proof of \eqref{eq:tinfty}, \eqref{eq:tinftyLinfty}]
Estimates \eqref{eq:tinfty} and \eqref{eq:tinftyLinfty} follow from a direct computation using convolution estimates.
\end{proof}
\begin{proof}[Proof of \eqref{eq:possionmaximal}]Estimate \eqref{eq:possionmaximal} follows from \cite[II, \textsection 2.1, Proposition, p. 57]{Stein-Harmonic-Analysis}, since $P_t^s f = p_t^{s} \ast f$, with
\[
 p_1^s(z) = c\frac{1}{\brac{1+|z|^2}^{\frac{n+s}{2}}}
\]
a kernel which is bounded, radial, and in $L^1(\R^n)$.
\end{proof}

\begin{proof}[Proof of \eqref{eq:possionmaximalpt}] For $s=1$ observe that $\partial_t P_t = c\laps{1} P_t$ (which follows from $\mathcal{F}(p_t^1)(\xi) = e^{-t|\xi|}$). Thus \eqref{eq:possionmaximalpt} follows from \eqref{eq:possionmaximal}.

The case $s \neq 1$ requires more work. We use the representation
\begin{equation}\label{eq:tptptfrep}
 t^{1-s} \partial_t P_t f(x) = c\int_{\R^n} (|x-z|^2 + t^2)^{\frac{2-s-n}{2}} \lap f(z)\ dz.
\end{equation}
To see \eqref{eq:tptptfrep}, one can use the Fourier representation in \eqref{eq:Fourierpoisson}. Alternatively, we solve an initial value problem for an ordinary differential equation: By \eqref{eq:spoissonpde}, 
\[
 \partial_{t} (t^{1-s} \partial_t P_t f) = - t^{1-s} P_t \lap_x f,
\]
so both sides of \eqref{eq:tptptfrep} solve the same equation. Moreover at $t=0$, both sides of \eqref{eq:tptptfrep} coincide: since $|x-z|^{2-s-n}$ is the kernel of the Riesz potential $\lapms{2-s}$,
\[
 \lim_{t\to 0} t^{1-s} \partial_t P_t f(x) = c\laps{s} f(x) = \lim_{t \to 0} c\int_{\R^n} (|x-z|^2 + t^2)^{\frac{2-s-n}{2}} \lap f(z)\ dz.
\]
The relation \eqref{eq:tptptfrep} is now established, since both sides of \eqref{eq:tptptfrep} solve the same equation in $t$ and have the same initial datum at $t=0$.

Now we set $g(z) := \laps{s} f(tz)$. Note that $\mathcal{M} g(x) = c\mathcal{M} \laps{s} f(x)$. \eqref{eq:tptptfrep} then follows once we can show
\begin{equation}\label{eq:goalmaximal}
 \sup_{|x-y| < 1} c\int_{\R^n} (|y-z|^2 + 1)^{\frac{2-s-n}{2}} \laps{2-s} g(z)\ dz \aleq \mathcal{M}g(x).
\end{equation}
To obtain \eqref{eq:goalmaximal} we use
\begin{equation}\label{eq:lapsgbessel}
 \laps{\gamma} (|x|^2 + 1)^{\frac{\gamma-n}{2}} =c (|x|^2 + 1)^{\frac{-\gamma-n}{2}} 
\end{equation}
This equation might look surprising at first -- in particular one might think its the wrong 'homogeneity' if one thinks of $\laps{\gamma}$ as the $s$-derivative. But $\laps{\gamma}$ behaves more like the Laplacian $\lap$, indeed we suggest to check this formula for $\gamma =2$. The computation \eqref{eq:lapsgbessel} can be found in Dyda, Kuznetsov, Kwasnicki's \cite[Corollary 1]{Dyda-Kuznetsov-Kwasnicki-2015} who compute several explicit fractional Laplacians in terms of the Meijer G-function. As the authors informed us, special cases of \eqref{eq:lapsgbessel} appear in the work of Samko, see for example \cite{Samko-2002}. It is also possible to obtain \eqref{eq:lapsgbessel} from the Fourier representation \eqref{eq:Fourierpoisson}.

To obtain \eqref{eq:goalmaximal} from \eqref{eq:lapsgbessel} simply integrate by parts
\[
\begin{split}
  &\sup_{|x-y| < 1} c\int_{\R^n} (|y-z|^2 + 1)^{\frac{2-s-n}{2}} \laps{2-s} g(z)\ dz\\
  =&\sup_{|x-y| < 1} c\int_{\R^n} (|y-z|^2 + 1)^{\frac{-(2-s)-n}{2}} g(z)\ dz.
 \end{split}
\]
Now the kernel $(|y-z|^2 + 1)^{\frac{-(2-s)-n}{2}}$ is bounded and belongs to $L^1$, so it falls into the realm of Stein's \cite[II, \textsection 2.1, Proposition, p. 57]{Stein-Harmonic-Analysis}. This proves \eqref{eq:possionmaximalpt}.
\end{proof}

\begin{proof}[Proof of \eqref{eq:possionmaximalptx}]
For \eqref{eq:possionmaximalptx}, a rougher estimate than \eqref{eq:lapsgbessel} suffices,
\begin{equation}\label{eq:roughestimate}
 \laps{1-s} (|x|^2 + 1)^{-\frac{n+s}{2}} \aleq (|x|^2 + 1)^{-\frac{n+1}{2}}.
\end{equation}
Then 
\[
\begin{split}
  &\sup_{|x-y| < 1} c\int_{\R^n} (|y-z|^2 + 1)^{\frac{-n-s}{2}} \nabla f(z)\ dz\\
  \aleq &\sup_{|x-y| < 1} c\int_{\R^n} (|y-z|^2 + 1)^{-\frac{n+1}{2}} |\nabla^s f(z)|\ dz.
 \end{split}
\]
Again we can conclude with Stein's \cite[II, \textsection 2.1, Proposition, p. 57]{Stein-Harmonic-Analysis}.
\end{proof}
\begin{proof}[Proof of \eqref{eq:maximalpotential}]
We have
\[
 t^\sigma P_t^s f = \int_{\R^n} t^{-n}\kappa\brac{\frac{x-y}{t}}\ \lapms{\sigma} f(y)\ dy,
\]
where
\[
 \kappa(z) = \laps{\sigma} \frac{1}{\brac{|z|^2 + 1}^{\frac{n+s}{2}}}.
\]
For $\sigma \geq 0$, $\kappa \in L^1(\R^n) \cap L^\infty(\R^n)$ is radial, and we conclude again with \cite[II, \textsection 2.1, Proposition, p. 57]{Stein-Harmonic-Analysis}.
\end{proof}

\subsection{Proof of Propositions~\texorpdfstring{\ref{pr:Sobolevspaces}}{\ref*{pr:Sobolevspaces}}, \texorpdfstring{\ref{pr:Lorentzspaces}}{\ref*{pr:Lorentzspaces}}, \texorpdfstring{\ref{pr:BMOcharacterization}}{\ref*{pr:BMOcharacterization}}}

\begin{proof}[Proof of Proposition~\ref{pr:Sobolevspaces}]
Proposition~\ref{pr:Sobolevspaces} follows from the Besov- and Triebel space characterization by Bui-Candy \cite[Theorem 1.1, Theorem 1.3.]{BuiCandy-2015}. To ensure the ``Cancellation condition (C1)'' in their article, one needs to use the growth estimates from Section~\ref{s:fouriertransform}. 
\end{proof}

\begin{proof}[Proof of Proposition~\ref{pr:Lorentzspaces}]
The claim follows by estimates on so-called (non-tangential) square functions. More precisely, we use \cite[Chapter I, \textsection 8.23, p.46]{Stein-Harmonic-Analysis}. There it is shown that
\[
\left \|x \mapsto \left |\int_{(y,t): |y-x|<t} t^{-1-n}|  q_t \ast f |^2\ dy\, dt \right |^{\frac{1}{2}} \right \|_{L^{(p,q)}(\R^n)} \aeq \|f\|_{L^{(p,q)}(\R^n)},
\]
where $q_t = t^{-n}q(z/t)$ and $q$ is suitably growing radial kernel, belongs to $L^\infty(\R^n) \cap L^1(\R^n)$, and $\int_{\R^n} q = 0$. In particular $t\partial_t p^s_t$ satisfies these conditions, and thus
\[
\left \|x \mapsto \left |\int_{(y,t): |y-x|<t} t^{1-n}|  \partial_t F^s(x,t) |^2\ dy\, dt \right |^{\frac{1}{2}} \right \|_{L^{(p,q)}(\R^n)} \aeq \|f\|_{L^{(p,q)}(\R^n)},
\]
More generally, with help of the representation \eqref{eq:tptptfrep}, we may find a suitable $q$ when $0 \leq \nu < s$ such that
\[
 t^{1-\sigma}\partial_t p_t^s \ast f = q_t \ast \laps{\sigma} f. 
\]
This leads to \eqref{eq:squarefctchar}. \eqref{eq:squarefctnablachar} follows by the same argument.
\end{proof}

\begin{proof}[Proof of Proposition~\ref{pr:BMOcharacterization}]
We refer to \cite[Theorem 3. Chapter IV, \textsection 4.3., p.159]{Stein-Harmonic-Analysis} together with the remark on the kernel in \cite[Chapter IV, \textsection 4.4.3., p.165]{Stein-Harmonic-Analysis}. 

This result can also be recovered via the Poisson characterizations of Triebel-Lizorkin spaces by \cite{BuiCandy-2015}, using the Triebel-Lizorkin space characterization of BMO.
\end{proof}

\subsection{Reflected harmonic extensions: Proof of Proposition~\ref{pr:BMOextension}}\label{s:BMOextension}
In \cite[Appendix 3]{Brezis-Nirenberg-1996} Brezis and Nirenberg, together with Mironescu, show that the harmonic extension of a $VMO$-function defined on the boundary of a bounded domain $\partial \Omega$ extends to a $VMO$-function in $\Omega$. Their definition of $BMO(\Omega)$, \cite[\textsection II.1, Definition 1, p.313]{Brezis-Nirenberg-1996}, however excludes balls that intersect the boundary. In the following we adapt their proof to our situation.

From now on we denote with $F(x,t) \equiv F^e(x,t) := P^1_{|t|} f$ the harmonic extension to $\R^{n+1}_+$ of $f: \R^n \to \R$ reflected evenly across $\R^n$. 

The first step is to replace $F$ by another function which easier to compute. For $x \in \R^n$, $t \in (0,\infty)$ we pick the ball $B_t(x) \subset \R^n$ and define
\[
 G(x,t) := \mvint_{B_t(x)} f \equiv (f)_{B_t(x)}.
\]
This is possible due to the embedding $L^\infty \subset BMO$ and the following Lemma, cf. \cite[Lemma A3.1.]{Brezis-Nirenberg-1996}.
\begin{lemma}\label{la:meanvaluecompare}
There is a uniform constant $c \in \R$ such that
\[
 \sup_{(x,t) \in \R^{n+1}} \left |c\ P_{|t|} f(x) - \mvint_{B_{|t|}(x)} f \right | \aleq [f]_{BMO(\R^n)}.
\]
In other words, 
\[
 \|F -G \|_{L^\infty(\R^{n+1})} \leq [f]_{BMO(\R^n)}.
\]
\end{lemma}
\begin{proof}
Pick $c := (P^1_t[1])^{-1} \in \R$ such that 
\[
|cP_t f(x) - (f)_{B_{t}(x)} | 
\aleq \int_{\R^n} \frac{t}{\brac{|x-z|^2 + t^2}^{\frac{n+1}{2}}} |f(z) - (f)_{B_t(x)}|\ dz.
\]
Now we split the integration domain into $B_{t}(x)$ and annuli
\[
 \leq \int_{B_t(x)} \frac{t}{\brac{|x-z|^2 + t^2}^{\frac{n+1}{2}}} |f(z) - (f)_{B_t(x)}|\ dz + \sum_{k=1}^\infty \int_{B_{2^kt}(x)\backslash B_{2^{k-1}t}(x) } \frac{t}{\brac{|x-z|^2 + t^2}^{\frac{n+1}{2}}} |f(z) - (f)_{B_t(x)}|\ dz.
\]
Estimating the kernel in these domains we have
\[
 \aleq \mvint_{B_t(x)} |f(z) - (f)_{B_t(x)}|\ dz + \sum_{k=1}^\infty 2^{-k}\mvint_{B_{2^kt}(x) }  |f(z) - (f)_{B_t(x)}|\ dz.
\]
On the first term we use the definition of BMO, in the second term we want to do the same and thus introduce $(f)_{B_{2^kt}(x)}$.
\[
 \aleq [f]_{BMO} + \sum_{k=1}^\infty 2^{-k}\mvint_{B_{2^kt}(x) }  |f(z) - (f)_{B_{2^kt}(x)}|\ dz + \sum_{k=1}^\infty 2^{-k} |(f)_{B_{2^kt}(x)} - (f)_{B_t(x)}|\ dz\\
\]
Now we can estimate the second term again with the BMO-term and the sum converges. For the third term we write the difference of mean values as a telescoping sum,
\[
 \aleq [f]_{BMO} + \sum_{k=1}^\infty \sum_{j=1}^k 2^{-k} | (f)_{B_{2^{j}t}(x)} - (f)_{B_{2^{j-1}t}(x)}|\\
\]
Again we estimate by the BMO-norm, and are left with
\[
\aleq [f]_{BMO} + \sum_{k=1}^\infty \sum_{j=1}^k 2^{-k} [f]_{BMO}.
\]
To see that this sum converges, we use the Fubini theorem for series. Namely,
\[
 \sum_{k=1}^\infty \sum_{j=1}^k 2^{-k} = \sum_{j=1}^\infty\sum_{k=j}^\infty  2^{-k} = \sum_{j=1}^\infty2^{-j} < \infty.
\]
\end{proof}

To measure the BMO-norm, in Definition~\ref{eq:BMOseminorm} one can replace the balls with other objects such as squares, cylinders. To this end, in $\R^{n+1}$ we consider the following cylinders
\[
 \tilde{D}^{n+1}_\rho(x_0,t_0), := B^n_{\rho}(x_0) \times (t_0-\rho,t_0+\rho) \quad x_0 \in \R^n, t_0 \in \R.
\]
The following estimate was proven in \cite[\textsection II.3, Lemma 7,  p.327]{Brezis-Nirenberg-1996}. It treats the case when the cylinder is away from the boundary.

\begin{lemma}[Estimates away from the boundary]\label{la:BMOawayest}
The following holds:
\[
\sup_{t_0 \in \R, x_0 \in \R^n} \sup_{\rho > 0:\, 2\rho < |t_0|} \mvint_{\tilde{D}^{n+1}_\rho(x_0,t_0)} |G - (G)_{\tilde{D}^{n+1}_\rho(x_0,t_0)} | \leq C_{n} \ [f]_{BMO(\R^n)}.
\]
\end{lemma}
\begin{proof}
Fix $x_0 \in \R^n$, $t_0 \in \R$ and $\rho > 0$. Set
\[
\mathcal{I} := \mvint_{\tilde{D}^{n+1}_\rho(x_0,t_0)} |G - (G)_{\tilde{D}^{n+1}_\rho(x_0,t_0)} | \aleq \mvint_{\tilde{D}^{n+1}_\rho(x_0,t_0)}\mvint_{\tilde{D}^{n+1}_\rho(x_0,t_0)} |G(|s_1|,x_1) - G(|s_2|,x_2)|\ d(s_1,x_1) d(s_2,x_2).
\]
For $s_1,s_2 \in (t_0-\rho,t_0+\rho)$, $x_1,x_2 \in B^n_\rho(x_0)$ we have
\[
 B^n_{|s_1|}(x_1),\ B^n_{|s_2|}(x_2) \subset B^n_{|t_0|+2\rho}(x_0).
\]
Consequently, in view of \cite[Lemma A.4, p. 36]{Brezis-Nirenberg-1995} which states that for $A \subset B$,
\[
 |\mvint_{A} g - \mvint_{B} g| \aleq \frac{|B|}{|A|} \mvint_{B} |g - \mvint_{B} g|,
\]
we have
\[
 |G(|s_1|,x_1) - (f)_{B^n_{|t_0|+2\rho}(x_0)}| \aleq \frac{|B^n_{|t_0|+2\rho}(x_0)|}{|B^n_{|s_1|}(x_1)|}\ [f]_{BMO(\R^n)}.
\]
With the assumption $|t_0| > 2\rho$,
\[
 \frac{|B^n_{|t_0|+2\rho}(x_0)|}{|B^n_{|s_1|}(x_1)|} \leq C_n \brac{2\frac{|t_0|+2\rho}{\max \{\rho,|t_0|\}}}^n  \leq 6^n\, C_n.
\]
Consequently,
\[
 |G(|s_1|,x_1) - (f)_{B^n_{|t_0|+2\rho}(x_0)}|,\ |G(|s_2|,x_2) - (f)_{B^n_{|t_0|+2\rho}(x_0)}| \aleq [f]_{BMO(\R^n)}.
\]
Plugging this into $\mathcal{I}$, we obtain $\mathcal{I} \aleq [f]_{BMO(\R^n)}$.
\end{proof}

Since we want to find an BMO-estimate up to (and over the) boundary $\R^n \times \{0\}$, we need to accompany Lemma~\ref{la:BMOawayest} with an estimate close to the boundary $\R^n \times \{0\}$. Namely we have
\begin{lemma}[Close to the boundary]\label{la:BMOcloseest}
For any $\Lambda > 0$ the following holds.
\[
\sup_{t_0 \in \R, x_0 \in \R^n}\, \sup_{\rho > 0:\, |t_0| \leq \Lambda \rho} \mvint_{\tilde{D}^{n+1}_\rho(x_0,t_0)} |G - (G)_{\tilde{D}^{n+1}_\rho(x_0,t_0)} | \leq C_{n} (\Lambda +2)^n\ [f]_{BMO(\R^n)}.
\]
\end{lemma}
\begin{proof}
Fix $x_0 \in \R^n$, $t_0 \in \R$ and $\rho > 0$. Set
\[
\begin{split}
\mathcal{I} :=& \mvint_{\tilde{D}^{n+1}_\rho(x_0,t_0)} |G - (G)_{\tilde{D}^{n+1}_\rho(x_0,t_0)} |\\
\aleq &\rho^{-2(n+1)} \int_{t_0-\rho}^{t_0+\rho}\int_{t_0-\rho}^{t_0+\rho} \int_{B_\rho(x_0)} \int_{B_\rho(x_0)} |G(y_1,s_1) - G(y_2,s_2)|\ dy_1\, dy_2\ ds_1\, ds_2.\\
\end{split}
\]
Now
\[
 |G(y_1,s_1) - G(y_2,s_2)| \leq \mvint_{B_1(0)}\mvint_{B_1(0)} |f(y_1 + |s_1| z_1) - f(y_2 + |s_2| z_2)|\ dz_1 \ dz_2.
\]
Consequently, with Fubini
\[
\begin{split}
&\int_{B_\rho(x_0)}\int_{B_\rho(x_0)} |G(y_1,s_1) - G(y_2,s_2)| dy_1\ dy_2\\
\aleq & \mvint_{B_1(0)}\mvint_{B_1(0)} \int_{B_\rho(x_0)}\int_{B_\rho(x_0)} |f(y_1 + |s_1| z_1) - f(y_2 + |s_2| z_2)|\ dy_1\ dy_2\ dz_1 \ dz_2.
\end{split}
 \]
Next, by substitution,
\[
 \begin{split}
&  \int_{B_\rho(x_0)}\int_{B_\rho(x_0)} |f(y_1 + |s_1| z_1) - f(y_2 + |s_2| z_2)|\ dy_1\ dy_2\\
=& \int_{B_\rho(x_0 + |s_1| z_1)}\int_{B_\rho(x_0+|s_2| z_2)} |f(y_1) - f(y_2 )|\ dy_1\ dy_2.
 \end{split}
\]
Now if $z_1,z_2 \in B_1(0)$, $s_1,s_2 \in (t_0-\rho,t_0+\rho)$ 
\[
 B_\rho(x_0 + |s_1| z_1), B_\rho(x_0+|s_2| z_2)  \subset B_{t_0 + 2 \rho} (x_0).
\]
We thus have
\[
 \mathcal{I} \aleq \rho^{-2n} \int_{B_{t_0 + 2\rho} (x_0)}\int_{B_{t_0 + 2\rho} (x_0)} |f(y_1) - f(y_2 )|\ dy_1\ dy_2,
\]
and with $|t_0| \leq \Lambda \rho$,
\[
 \aleq \rho^{-2n} \int_{B_{(\Lambda + 2)\rho} (x_0)}\int_{B_{(\Lambda + 2)\rho}(x_0)} |f(y_1) - f(y_2 )|\ dy_1\ dy_2.
\]
Finally we use the definition of BMO, \eqref{eq:BMOseminorm}, and have
\[
\aleq  (\Lambda+2)^{n}\ [f]_{BMO}.
 \]
\end{proof}

Now we have all the ingredients for the proof of Proposition~\ref{pr:BMOextension}:
\begin{proof}
From Lemma~\ref{la:BMOawayest} and Lemma~\ref{la:BMOcloseest} we obtain
\[
 [G]_{BMO(\R^{n+1})} \aleq [f]_{BMO(\R^n)}.
\]
With the help of the embedding $L^\infty \subset BMO$ and Lemma~\ref{la:meanvaluecompare} we obtain
\[
  [F^e]_{BMO(\R^{n+1})} \leq 2\|F^e - G\|_{L^\infty} + [G]_{BMO(\R^{n+1})} \aleq [f]_{BMO(\R^n)}.
\]
Proposition~\ref{pr:BMOextension} is proven.
\end{proof}

\subsection{Proof of Proposition~\texorpdfstring{\ref{pr:hoelder}}{\ref*{pr:hoelder}}}
\begin{proof}[Proof of \eqref{eq:Hoelderequivalence}]
In Stein's \cite[V, \textsection 4.2, Proposition 7, p.142]{Stein-Singular-Integrals} the following is proven for any $\nu < s$.
\[
 \|f\|_{\infty} + \sup_{t > 0} \sup_{x \in \R^n}| |t^{1-\nu} \partial_t P^s_t f| \aeq \|f\|_{\infty}  + [f]_{C^{0,\nu}(\R^n)}
\]
Indeed, it is proven for $P^1_t f$, but this easily extends to $P^s_t f$. Apply this equation to $f_k(x) := k^{-\nu} f(kx)$, and one has
\[
 k^{-\nu} \|f\|_{\infty} + \sup_{t > 0} \sup_{x \in \R^n}| |t^{1-\nu} \partial_t P_t f| \aeq k^{-\nu} \|f\|_{\infty}  + [f]_{C^{0,\nu}(\R^n)}.
\]
Letting $k \to \infty$, we obtain \eqref{eq:Hoelderequivalence}.
\end{proof}
\begin{proof}[Proof of \eqref{eq:Hoelderequivalencelimit}]
Since $\brac{|\cdot|^2 + 1}^{-\frac{n+s}{2}}$ is integrable, with H\"older inequality
\[
 |\nabla_x P^s_t f(x)| \aleq \int_{\R^n} \frac{t^s}{\brac{|x-z|^2 + t^2}^{\frac{n+s}{2}}} |\nabla f(z)|\ dz \aleq \|\nabla f\|_{L^\infty}.
\]
This shows \eqref{eq:Hoelderequivalencelimit} for $\nu = 1$.

Let us now more generally consider any $0 < \nu \leq 1$, 
\[
 |\nabla_x P^s_t f(x)| = \left |\int_{\R^n} t^s \nabla_z \brac{|x-z|^2 + t^2}^{-\frac{n+s}{2}}  f(z) dz \right |.
\]
Now observe that $\int_{\R^n} \nabla_z \brac{|x-z|^2 + t^2}^{-\frac{n+s}{2}} = 0$, and thus
\[
 = \left |\int_{\R^n} t^s |x-z|^\nu \nabla_z \brac{|x-z|^2 + t^2}^{-\frac{n+s}{2}}  \frac{f(z)-f(x)}{|x-z|^\nu} dz \right |.
\]
Since $\nu \leq 1$,
\[
 \aleq  [f]_{C^\nu}\ \left |\int_{\R^n} t^s |z|^\nu \nabla_z \brac{|z|^2 + t^2}^{-\frac{n+s}{2}}  dz \right |.
\]
If we set $\kappa(z) := |z|^\nu \nabla_z \brac{|z|^2 + 1}^{-\frac{n+s}{2}}$,
\[
 = [f]_{C^\nu}\ t^{\nu-1}\ \|\kappa \|_{L^1(\R^n)}.
\]
Since $\kappa$ is integrable whenever $\nu < 1+s$, we have shown
\[
 |\nabla_x P^s_t f(x)| \aleq t^{\nu-1}\ [f]_{C^\nu}. 
\]
Proposition~\ref{pr:hoelder} is proven. 
\end{proof}

\subsection{On Theorem~\texorpdfstring{\ref{th:BTcharacterization}}{\ref*{th:BTcharacterization}}}
The theorem follows from the work by Bui and Candy \cite[Theorem 1.1, Theorem 1.3.]{BuiCandy-2015}. As mentioned above, in particular the ``Cancellation condition (C1)'' has to be ensured, but this can be checked with the explicit representation and estimates for the Fourier transform of the Poisson kernel in Section~\ref{s:fouriertransform}. \qed

\section{Proofs and Literature for Section~\texorpdfstring{\ref{s:bbestimates}}{\ref*{s:bbestimates}}}\label{s:proofbbestimates}
We need an extension of the $L^\infty$-$L^1$-H\"older-inequality on $\R^{n+1}_+$. We have the following estimate between Carleson-measures and square functions, which can be found in \cite[IV, \textsection 4.4, Proposition, p. 162]{Stein-Harmonic-Analysis}. 
\begin{lemma}\label{la:FGcarlesonest}
\[
 \int_{\R^{n+1}_+} F(x,t)\, G(x,t)\, d(x,t)\aleq
 \]
 \[\sup_{B \subset \R^n \mbox{ balls} } \brac{|B|^{-1} \int_{T(B)} t| F(y,t)|^2 dy\, dt}^{\frac{1}{2}}\ \int_{\R^n} \brac{\int_{|x-y|<t} |G(y,t)|^2 \frac{dy dt}{t^{n+1}} }^{\frac{1}{2}}dx
 \]
where the supremum is over balls $B \subset \R^n$ and $T(B)$ is the ``tent'' over $B$ in $\R^n$, i.e. $T(B_r(x_0)) = \{(x,t) \in \R^{n+1}_+: |x-x_0| < r-t\}$.
\end{lemma}
If $F$ is chosen to be $\nabla_{\R^{n+1}} P^s_t f$, then we can use the BMO-characterization of $f$ in Proposition~\ref{pr:BMOcharacterization}, and have the following corollary.
\begin{corollary}\label{co:FGcarlesonestBMO}
Let $F^s(x,t) := P^s_t f(x)$, $s \in (0,2)$. Then
\[
 \int_{\R^{n+1}_+} |\nabla_{\R^{n+1}} F^s(x,t)|\, |G(x,t)|\, d(x,t)\aleq [f]_{BMO}\ \int_{\R^n} \brac{\int_{|x-y|<t} |G(y,t)|^2 \frac{dy dt}{t^{n+1}} }^{\frac{1}{2}}dx.
 \]
\end{corollary}

\begin{proof}[Proof of Proposition~\ref{pr:Lpqest}]
For the first two estimates, observe that with \eqref{eq:maximalpotential}
\[
 \sup_{t > 0} t^{s_3} |H(x,t)| \aleq \mathcal{M} |\lapms{s_3} h|(x).
\]
The estimates then follow from H\"older's inequality and \eqref{eq:Hspequivalencenablax}, \eqref{eq:Hspequivalencenablaxx}, \eqref{eq:Hspequivalencedt}, respectively.
\end{proof}

\begin{proof}[Proof of Proposition~\ref{pr:BMOusualestimate}]
Denote the square function
\[
 \mathcal{S}(H)(x) := \brac{\int_{|x-y|<t} |H(y,t)|^2 \frac{dy dt}{t^{n+1}} }^{\frac{1}{2}}.
\]
Let $W(x,t) := t^{2-s}\nabla_{x}\nabla_{\R^{n+1}} G^s$. Then, in view of Corollary~\ref{co:FGcarlesonestBMO} and \eqref{eq:possionmaximalpt},
\[
\begin{split}
&\int_{\R^{n+1}_+} t^{1+(1-s)(\ell+1)} |\nabla_{\R^{n+1}} \Phi^s(x,t)|\, |\nabla_{x} \nabla_{\R^{n+1}} G^s(x,t)|\, |\partial_t F_1^s(x,t)|\ldots |\partial_t F_\ell^s(x,t)|\, d(x,t) \\
\aleq & [\varphi]_{BMO}\, \int_{\R^{n}} \mathcal{M}(\laps{s} f_1)(x)\, \ldots \mathcal{M}(\laps{s} f_\ell)(x)\, \mathcal{S}(W)(x)  
\end{split}
\]
If we had to work with $\nabla_{\R^{n+1}} F_i^s(x,t)$ then we would also use \eqref{eq:possionmaximalptx} and obtain the same only with $\mathcal{M}(\laps{s} f_i)(x)$ replaced by $\mathcal{M}(\laps{s} f_i)(x)+\mathcal{M}(\nabla^s f_i)(x)$.

With H\"older's inequality and the maximal theorem (here we use that $p_i \in (1,\infty)$),
\[
 \aleq [\varphi]_{BMO}\, \| \laps{s} f_1 \|_{L^{(p_1,q_1)}}\, \ldots \| \laps{s} f_\ell \|_{L^{(p_\ell,q_\ell)}} \| \mathcal{S}(W) \|_{L^{(p_0,q_0)}}.
\]
We conclude the first estimate of Proposition~\ref{pr:BMOusualestimate} with \eqref{eq:squarefctnablachar}. The other estimates follow the same way.
\end{proof}

\begin{proof}[Proof of Proposition~\ref{pr:Hoelderusualestimate}]
First assume that $s_2 > 0$. In view of \eqref{eq:Hoelderequivalence} and \eqref{eq:Hoelderequivalencelimit},
\[
\mathcal{I} := \int_{\R^{n+1}_+} t^{1-\nu-s_1+s_2} |\nabla_{\R^{n+1}} \Phi^s(x,t)|\, |\nabla_{\R^{n+1}} G^s(x,t)|\, |F^s(x,t)|\, d(x,t) 
\]
\[
\aleq  [\varphi]_{C^\nu}\, \int_{\R^{n+1}_+} t^{s_2-s_1} |\nabla_{\R^{n+1}} G^s(x,t)|\, |F^s(x,t)|\, d(x,t). 
\]
With H\"older inequality we find
\[
\aleq  [\varphi]_{C^\nu}\, \left \| x \mapsto \brac{\int_{t=0}^\infty |t^{\frac{1}{2}-s_1} \nabla_{\R^{n+1}} G^s(x,t)|^2\ dt}^{\frac{1}{2}} \right \|_{L^{(p_2,q_2)}}\, \left \|x\mapsto \brac{\int_{t=0}^\infty  |t^{s_2-\frac{1}{2}}\, F^s(x,t)|^2\ dt}^{\frac{1}{2}} \right \|_{L^{(p_2,q_2)}},
\]
and with another H\"older inequality and in view of \eqref{eq:Hspequivalencedt} and \eqref{eq:Hspequivalencenablax},
\[
 \aleq [\varphi]_{C^\nu}\, \|\laps{s_1} g\|_{L^{(p_1,q_1)}}\,\left  \|x\mapsto \brac{\int_{t=0}^\infty  |t^{s_2-\frac{1}{2}}\, F^s(x,t)|^2\ dt}^{\frac{1}{2}} \right \|_{L^{(p_2,q_2)}}.
\]
Now 
\[
 t^{s_2-\frac{1}{2}}\, F^s(x,t) = t^{s_2-\frac{1}{2}}\ p_t^s \ast f(x) = t^{s_2-\frac{1}{2}}\ \laps{s_1} (p_t^s) \ast \lapms{s_1}f(x) = t^{-\frac{1}{2}} \kappa_t\ast \lapms{s_2} f(x),
\]
for
\[
 \kappa(z) = \laps{s_2} p_1^s(z),
\]
and $\kappa_t(z) = t^{-n}\kappa(z/t)$. Observe that $\kappa$ is integrable and $\int_{\R^n} \kappa = 0$, since $s_2 > 0$ and thus $\kappa$ is a (fractional) derivative. Thus,
\[
 \left \|x\mapsto \brac{\int_{t=0}^\infty  |t^{s_2-\frac{1}{2}}\, F^s(x,t)|^2\ dt}^{\frac{1}{2}} \right \|_{L^{(p_2,q_2)}} = \left \|x\mapsto \brac{\int_{t=0}^\infty  |\kappa_t \ast \lapms{s_2} f(x)|^2\ \frac{dt}{t}}^{\frac{1}{2}} \right \|_{L^{(p_2,q_2)}}
\]
This is a (tangential) square function as in \cite[Chapter I, \textsection 6.3, (20), p. 27]{Stein-Harmonic-Analysis}, and with \cite[Chapter I, \textsection 8.23, p. 46]{Stein-Harmonic-Analysis} we conclude
\[
  \left \|x\mapsto \brac{\int_{t=0}^\infty  |t^{s_2-\frac{1}{2}}\, F^s(x,t)|^2\ dt}^{\frac{1}{2}} \right \|_{L^{(p_2,q_2)}} \aleq \|\lapms{s_2} f\|_{L^{(p,q)}}.
\]
This shows the main estimate in Proposition~\ref{pr:Hoelderusualestimate} for $s_2 > 0$, the other estimates for $s_2 > 0$ follow by variations of the above argument. If $F^s(x,t)$ is replaced with $t|\nabla_{\R^{n+1}} F^s(x,t)$ and $s_2 = 0$ we still can argue as above and find a suitable $\kappa$ with the mean value property $\int_{\R^n} \kappa = 0$.

The remaining case $s_2 = 0$ follows essentially as in the proof of Proposition~\ref{pr:BMOusualestimate}. One needs to replace the BMO-characterization in \eqref{eq:BMOequivalence} with
\[
 \sup_{B \subset \R^n} \brac{|B|^{-1}\int_{T(B)}  t^{1-2\nu}|\nabla_{y} F^s(y,t)|^2\, dy\, dt}^{\frac{1}{2}} \aleq [D^\nu f]_{BMO(\R^n)},
\]
and
\[
 \sup_{B \subset \R^n} \brac{|B|^{-1}\int_{T(B)}  t^{1-2\nu}|\partial_t F^s(y,t)|^2\, dy\, dt}^{\frac{1}{2}} \aleq [\laps{\nu} f]_{BMO(\R^n)}.
\]
The latter holds for $0 \leq \nu < \min \{s,1\}$ by the same arguments that lead to \eqref{eq:BMOequivalence}, see \cite[Theorem 3. Chapter IV, \textsection 4.3., p.159f]{Stein-Harmonic-Analysis}: one writes
\[
 t^{-\nu}\nabla_{y} F^s(y,t) = t^{-1} \kappa_t \ast D^s f(y), \quad t^{-\nu} \partial_t F^s(y,t) = t^{-1} \kappa_t \ast \laps{s} f(y),
\]
for suitable kernels $\kappa \in L^1(\R^n)$ with the mean value property $\int_{\R^n} \kappa = 0$. Observe that for $f \in C_c^\infty(\R^n)$, by boundedness of the Riesz transform on $BMO$,
\[
 [D^\nu f]_{BMO(\R^n)} \aeq [\laps{\nu} f]_{BMO(\R^n)}.
\]
\end{proof}

\end{appendix}

\subsection*{Acknowledgment}
We would like to thank P. Haj\l{}asz for pointing out to us the work by Brezis and Nguyen~\cite{Brezis-Nguyen-2011}. 

E. Lenzmann is supported by the Swiss National Science Foundation (SNF) through Grant No. 200021–149233. 
A. Schikorra is supported by the German Research Foundation (DFG) through Grant No. SCHI-1257-3-1. A.S. is Heisenberg fellow.

\bibliographystyle{amsplain}
\bibliography{bib}%

\end{document}